\documentclass[a4paper, 11pt]{amsart}
\usepackage{amsmath,amssymb,graphicx, amscd, pstricks, hhline,
enumerate, xcolor, pstricks-add}
\usepackage{amsthm}
\usepackage{listings, enumitem, tikz}
\usetikzlibrary{matrix,arrows}
\usepackage[latin1]{inputenc}
\usepackage{url}
\usepackage{chngcntr}
 \usepackage{setspace} 
\DeclareMathOperator{\im}{im}

\DeclareMathOperator{\coker}{coker}
\DeclareMathOperator{\ob}{ob}

\DeclareMathOperator{\Hom}{Hom}
\DeclareMathOperator{\Mor}{Mor}

\DeclareMathOperator{\Mod}{\R-\sf Mod}

\DeclareMathOperator{\GABan}{(\G,\A)-\sf Ban}

\DeclareMathOperator{\DeckMod}{\Deck(\mathnormal{X})-\sf Mod}
\DeclareMathOperator{\normedmod}{\R-\sf Mod_{\| \cdot\|}}
\DeclareMathOperator{\normedmodone}{\R-\sf Mod_{\| \cdot\|}^1}
\DeclareMathOperator{\normedgmod}{\G-\sf Mod_{\| \cdot\|}}
\DeclareMathOperator{\normedamod}{\A-\sf Mod_{\| \cdot\|}}
  \DeclareMathOperator{\Grp}{{\sf Grp}}

  \DeclareMathOperator{\GrpBan}{{\sf GrpBan}}
  \DeclareMathOperator{\GrpBanc}{{\sf Grp\overline{Ban}}}
  \DeclareMathOperator{\GrppBan}{{\sf Grp^2Ban}}
  \DeclareMathOperator{\GrppBanc}{{\sf Grp^2\overline{Ban}}}
  \DeclareMathOperator{\TopBanc}{{\sf Top\overline{Ban}}}
  \DeclareMathOperator{\ToppBanc}{{\sf Top^2\overline{Ban}}}

  \DeclareMathOperator{\Modgrn}{{\sf \R-Mod_{\ast}^{\|\cdot\|}}}
  \DeclareMathOperator{\RCh}{{\sf{}_{\R}\! Ch}}
  \DeclareMathOperator{\GACh}{{\sf{}_{(\G,\A)}\! Ch^{\|\cdot\|}}}

  \DeclareMathOperator{\Top2}{{\sf Top^2}}

\DeclareMathOperator{\Alt}{Alt}
\DeclareMathOperator{\sgn}{sgn}
\DeclareMathOperator{\sing}{sing}
\DeclareMathOperator{\Deck}{Deck}
\DeclareMathOperator{\id}{id}

  \newcommand{\N}{\ensuremath{\mathbb{N}}}

  \newcommand{\R}{\ensuremath{\mathbb{R}}}

  \newcommand{\G}{\ensuremath{\mathcal{G}}}
  \newcommand{\A}{\ensuremath{\mathcal{A}}}
  \newcommand{\B}{\ensuremath{\mathcal{B}}}
  \newcommand{\HG}{\ensuremath{\mathcal{H}}}
  
  \newcommand*{\longhookrightarrow}{\ensuremath{\lhook\joinrel\relbar\joinrel\rightarrow}}

 \newcommand{\fa}[1]{\forall_{#1}\;\;\;}

\setenumerate[1]{label = (\roman*)}
\begin{document}
\title{Relative Bounded Cohomology for Groupoids}
\author{Matthias Blank}
\date{\today}
\renewcommand{\labelenumi}{(\roman{enumi})}
\newtheorem{Satz}{Satz}[section]
\newtheorem{thm}[Satz]{Theorem}
\newtheorem*{thm*}{Theorem}
\newtheorem{thmi}{Theorem}
\newtheorem{prop}[Satz]{Proposition}
\newtheorem{kor}[Satz]{Korollar}
\newtheorem{con}[Satz]{Consequence}
\newtheorem{notation}[Satz]{Notation}
\newtheorem*{Ziel}{Ziel}
\newenvironment{bew}{\begin{proof}[Beweis:]}{\end{proof}}
\theoremstyle{definition}
\newtheorem{qst}[Satz]{Question}
\newtheorem{exa}[Satz]{Example}
\newtheorem{def.}[Satz]{Definition}
\newtheorem*{def.*}{Definition}
\newtheorem*{out*}{Outlook}

\newtheorem*{mot}{Motivation}
\newtheorem*{rec}{Recall}
\newtheorem*{as}{Assumption}
\newtheorem{cor}[Satz]{Corollary}
\newtheorem*{cor*}{Corollary}
\newtheorem{rem}[Satz]{Remark}
\newtheorem{lemma}[Satz]{Lemma}

\begin{abstract}
We introduce bounded cohomology for (pairs of) groupoids and develop homological algebra to deal with it. We generalise results of Ivanov, Frigerio and Pagliantini to this setting and show that (under topological conditions) the bounded cohomology of a pair of topological spaces is isometrically isomorphic to the bounded cohomology of the pair of fundamental groupoids. Furthermore, we prove that bounded cohomology relative to an amenable groupoid is isometrically isomorphic to the bounded cohomology of the ambient groupoid. 
\end{abstract}
\maketitle
\section{Introduction}

Bounded cohomology is a functional-analytic variant of regular cohomology, first developed into an extensive theory with many applications by Gromov~\cite{Gr82}. 
To construct bounded cohomology, instead of looking at general cochains as for (singular or group) cohomology, one only considers \emph{bounded} cochains with respect to the canonical $\ell^1$-norm regarding the basis of simplices in the chain complex. The cohomology of the bounded cochain complex together with the induced semi-norm is then called bounded cohomology and denoted by $H_b^{\ast}$. The importance of the semi-norm can for instance be explained by the fact that, via the Hahn-Banach Theorem, bounded cohomology of a manifold relative to its boundary encompasses the so called \emph{simplicial volume}.  Via this relation, bounded cohomology enters Gromov's proof of Mostow's rigidity theorem~\cite{Gr82, Mu80} and more recently, the streamlined proof by Bessi\`eres, Besson, Boileau, Maillot and Porti of Perelman's geometrisation theorem~\cite{BBBMP10}. Furthermore, bounded cohomology can be used to show various super rigidity results~\cite{MS04,MMS04,CFI14,BF02,BBI13}. Finally, bounded cohomology is 
deeply related to geometric properties of groups and can for example characterise both amenable~\cite{Jo72,No90} and (relatively) hyperbolic groups~\cite{Mi02,Fr15}. 

Bounded cohomology behaves very differently from regular cohomology. Surprisingly for instance, bounded cohomology of a space essentially only depends on the fundamental group:
\begin{thm*}[The Mapping Theorem, {\cite[Theorem 4.1]{I}\cite{Gr82,Bue11}}]
Let $X$ be a connected CW-complex. Then for each $x\in X$, the classifying map induces an isometric isomorphism
 $H_b^{\ast}(X;\R)\longrightarrow H_b^{\ast}(\pi_1(X,x);\R)$.
\end{thm*}
The mapping theorem and its applications have been one of the hallmarks in the subject. It implies for instance that the simplicial volume of a space with an amenable fundamental group vanishes. In order to study manifolds with boundary, it would be useful to have also a relative version of the mapping theorem. Frigerio and Pagliantini~\cite{FP} have shown such a theorem for pairs of connected spaces under additional topological conditions. One goal of this article is to extend the mapping theorem also to the non-connected setting, so that we can apply the mapping theorem to manifolds with non-connected boundaries:
 \begin{thmi}[Relative Mapping Theorem, Theorem~\ref{t:relmap}]\label{t:intro1} Let $i\colon A\longhookrightarrow X$ be a CW-pair, such that $i$ is $\pi_1$-injective and induces isomorphisms between the higher homotopy groups on each connected component of $A$. Let $V$ be a Banach $\pi_1(X)$-module. Then there is a canonical isometric isomorphism 
\begin{align*}
 H_b^\ast(X,A;V')&\longrightarrow H_b^\ast(\pi_1(X),\pi_1(A);V').
\end{align*}
\end{thmi}

To make sense of the theorems right-hand side for non-connected spaces, this article will introduce bounded cohomology for pairs of groupoids. This is a natural way to define bounded cohomology relative to families of subgroups and allows us to easily extend many results about groups directly to groupoids. In particular, we will develop a homological algebra framework in the spirit of Ivanov to study resolutions that can calculate bounded cohomology for groupoids and show the following:
\begin{thmi}[Corollary~\ref{c:strongrelin}]\label{t:intro2}
 Let $\G$ be a groupoid and $V$ a Banach $\G$-module. Moreover, let~$((D^n,\delta^n_D)_{n\in\N}, \varepsilon \colon V\longrightarrow D^0)$ be a strong relatively injective~\G-resolution of $V$. Then there exists a canonical semi-norm non-increasing isomorphism of graded $\R$-modules
\[
 H^\ast (D^{\ast\G}) \longrightarrow H^\ast_b(\G;V).
\]
\end{thmi}
Using this result, we can straightforwardly generalise the mapping theorem to groupoids: 
\begin{thmi}[Absolute Mapping Theorem for Groupoids, Corollary~\ref{c:absmappingthm}]\label{t:intro3}
Let $X$ be a CW-complex and let $V$ be a Banach~$\pi_1(X)$-module. Then there is a canonical isometric isomorphism of graded semi-normed $\R$-modules
\begin{align*}
 H_b^{\ast}(X;V')\longrightarrow H_b^{\ast}(\pi_1(X);V')
\end{align*}
\end{thmi}
Similarly, we study pairs of resolutions that can calculate bounded cohomology of pairs of groupoids, show a fundamental lemma in this setting and derive the following:
\begin{thmi}[Corollary~\ref{c:relviastronginj}]\label{t:intro4}
 Let $i\colon \A\longrightarrow \G$ be a groupoid pair and~$V$ a Banach $\G$-module. Let $(C^\ast,D^\ast,\varphi^\ast,(\nu,\nu'))$ be a strong, relatively injective $(\G,\A)$-resolu\-tion~of $V$. Then there exists a canonical, semi-norm non-increasing isomorphism of graded $\R$-modules.
\[
 H^\ast(C^\ast,D^\ast,\varphi^\ast) \longrightarrow H^\ast_b(\G,\A;V).
\]
\end{thmi}
Finally, we will study amenable groupoids, show that they can be characterised in terms of bounded cohomology (Corollary~\ref{c:amenablevanish} and Proposition~\ref{p:amenablebounded}) and prove the following theorem:
\begin{thmi}[Algebraic Mapping Theorem, Corollary~\ref{c:algebraicmappingthm}]\label{t:intro5}
 Let $i\colon \A\longhookrightarrow\G$ be a pair of groupoids such that $\A$ is amenable. Let $V$ be Banach $\G$-module. Then for each $n\in \N_{\geq 2}$ there is a canonical isometric isomorphism
\[
H_b^n(\G,\A;V') \longrightarrow H_b^n(\G;V').
\]

\end{thmi}

\subsection*{Structure of the Article}
In \emph{Section 2}, we repeat basic facts about grou\-poids. Then, in \emph{Section 3} we introduce bounded cohomology for (pairs of) groupoids and show that this is an additive homotopy invariant and gives raise to a long exact sequence of bounded cohomology groups. In \emph{Section 4}, we develop relative homological algebra in our setting, introduce strong, relatively injective resolutions over (pairs of) groupoids and derive Theorem~\ref{t:intro2} and~\ref{t:intro4}. Then, in \emph{Section 5} we discuss amenable groupoids before proving in \emph{Section 6}  the algebraic mapping theorem.  In \emph{Section 7}, we deal with resolutions coming from topological spaces. We use this in \emph{Section~8} to slightly generalise the definition of bounded cohomology with twisted coefficients for pairs of topological spaces and derive the absolute mapping theorem. Finally, in \emph{Section 9}, we prove the relative mapping theorem. 
\subsection*{Acknowledgments}
At the time of the writing of this article, the author has been supported by the SFB 1085 \emph{Higher Invariants} (funded by the DFG) at the University of Regensburg. The article is based on parts of the authors Ph.D. thesis~\cite{Bl14} at the University of Regensburg, partially supported by the GRK 1692 \emph{Curvature, Cycles and Cohomology} (also funded by the~DFG). 
\section{Groupoids}
\subsection{Basics about Groupoids}
Groupoids are a generalisation of groups (and group actions), akin to considering not necessarily connected spaces in topology. They can be viewed as group-like structures where composition is only partially defined. 

For the convenience of the reader, in this section we repeat some basic facts and definitions about groupoids. We refer to the book of Ronald Brown~\cite{RB} for a more detailed overview. 
\begin{def.}\label{d:groupoids}
 \hfill 
\begin{enumerate}
 \item A \emph{groupoid} is a small category in which every morphism is invertible. 
\item A functor between groupoids is also called a \emph{groupoid map}.\label{i:d:groupoidsi} 
\item We will write $\Grp$ for the category of groupoids with groupoid maps as morphisms. 
\item In analogy to the group case, we consider morphisms as \emph{elements} of the groupoid. Hence, if $\G$ is a groupoid, we write $g\in \G$ to indicate that~$g\in \amalg_{e,f\in \ob \G} \Mor_{\G}(e,f)$ is a morphism in $\G$. 
\item A groupoid map is called \emph{injective/surjective} if it is injective/surjective on both objects and morphisms.
\label{i:d:groupoidsiv} 
\item Suppose $f,g\colon \G\longrightarrow \HG$ are groupoid maps. A natural equivalence between~$f$ and $g$ is also called a \emph{homotopy between $f$ and $g$}. If such a homotopy exists, we sometimes write $f\simeq g$. 
\item A groupoid $\G$ is called \emph{connected}, if for each pair $i,j\in \ob \G$ there exists at least one morphism from $i$ to $j$ in $\G$ (that is, if the underlying graph of the category $\G$ is connected). Similarly, we get the notion of \emph{connected components} of a groupoid. 
\end{enumerate}
\end{def.}
\newpage
\begin{exa}
 \hfill
\begin{enumerate}
 \item A group is naturally a groupoid with exactly one object. More precisely, we can identify the category of groups with the full subcategory of the category of groupoids, having the vertex set $\{1\}$.
\item In this sense, notions~\ref{i:d:groupoidsi} to~\ref{i:d:groupoidsiv} in Definition~\ref{d:groupoids} correspond to the concepts for groups. Two group homomorphisms~$f,g \colon G\longrightarrow H$ are homotopic if and only if there exists an inner automorphism $\alpha$ of $H$, such that $\alpha\circ f= g$. 
\item Given a family $(\G_i)_{i\in I}$ of groupoids, the \emph{disjoint union of $(\G_i)_{i\in I}$} is the groupoid $\amalg_{i\in I} \G_i$, defined by setting $\ob  \amalg_{i\in I} \G_i:= \amalg_{i\in I} \ob \G_i$ and for all $k,l\in I$, all $e\in \ob \G_k$ and all $f\in \ob \G_l$
\[
\Mor_{\amalg_{i\in I}\G_i}(e,f) := \begin{cases}
                                    \Mor_{\G_k}(e,f)&\text{if } k=l\\
				    \emptyset &\text{else,}
                                   \end{cases}
\]
together with the composition induced by the compositions of $\G_i$.
 In this fashion, we can view a family of groups naturally as a groupoid.  
\end{enumerate}
\end{exa}

\begin{exa}[Group Actions and Groupoids]\label{e:actiongrp}
 Let $G$ be a group and $X$ a set with a left $G$-action. We define a groupoid $G\ltimes X$, called \emph{the action groupoid or the semi-direct product of $X$ and $G$}, by setting:
\begin{enumerate}
 \item The objects of $G\ltimes X$ are given by $\ob G\ltimes X= X$.
\item For each $e,f\in X$, set $\Mor_{G\ltimes X}(e,f) =\{(e,g)\in X\times G\mid g\cdot e =f\}$.
\item Define the composition by setting for each $x\in X$ and $g,h\in G$
\[
 (g\cdot x,h)\circ (x,g) = (x,h\cdot g). 
\]
\end{enumerate}
\end{exa}

Since we view groupoids as groups where the composition is only partially defined, in the sense that we can only compose two elements if the target of the first matches the source of the second, it will be useful to define:
\begin{def.}
Let $\G$ be a groupoid. We define the \emph{source} and {target} maps by
\begin{align*}
 s\colon\G &\longrightarrow \ob \G\\
 \Mor_{\G}(e,f)\ni g &\longmapsto e\\
 t\colon\G &\longrightarrow \ob \G\\
 \Mor_{\G}(e,f)\ni g &\longmapsto f.
\end{align*}
\end{def.}

As the next theorem shows, up to homotopy we can actually always restrict to disjoint unions of groups:
\begin{thm}[Classifying groupoids up to homotopy]\label{t:classify}
 Let $\G$ be a groupoid and $i\colon \HG \longrightarrow \G$ be the inclusion of a full subgroupoid meeting each connected component of $\G$. Then there exists a groupoid map $p\colon \G \longrightarrow \HG$, such that 
\[
 p\circ i = \id_\HG \quad \text{ and } \quad i\circ p \simeq \id_\G.
\]
In particular, $\HG$ and $\G$ are equivalent. 
\end{thm}
\begin{proof}
 This is a well-known result~\cite[Theorem 3.1.7]{Bl14}.
\end{proof}
\newpage
\begin{cor}\label{c:grphequiv}
\hfill
\begin{enumerate}
\item Two groupoids $\HG$ and $\G$ are equivalent if and only if there is a bijection $\alpha\colon \Lambda(\G)\longrightarrow \Lambda(\HG)$, such that for all $e\in \Lambda\G$ the vertex groups $\G_e$ and $\HG_{\alpha(e)}$ are isomorphic. Here, we write $\Lambda(\G)$ and $\Lambda(\HG)$ to denote a choice of exactly one vertex in each connected component of $\G$ and $\HG$ respectively. 
 \item In particular: Every connected non-empty groupoid is equivalent to any of its vertex groups (which coincide up to isomorphism). 
\end{enumerate}
\end{cor}
\begin{proof}
Let $\varphi:= (\varphi_e\colon \G_e \longrightarrow \HG_{\alpha(e)})_{e\in \pi_0(\G)}$ be a family of group isomorphisms. Then $\varphi$ induces an isomorphism $\amalg_{e\in\pi_0(\G)} \G_e \longrightarrow \amalg_{e\in\pi_0(\HG)} \HG_e$ and hence
\[
 \G \simeq \amalg_{e\in\pi_0(\G)} \G_e \cong \amalg_{e\in\pi_0(\HG)} \HG_e \simeq \HG. \qedhere
\]
\end{proof}

\begin{def.}
 A \emph{pair of groupoids} is an injective groupoid map $i\colon \A\longrightarrow \G$. A map between pairs of groupoids $i\colon \A \longrightarrow \G$ and $j\colon \B\longrightarrow\HG$ is a pair of groupoid maps $(j,j')\colon (\G,\A)\longrightarrow (\HG,\B)$ commuting with $i$ and $j$.  
\end{def.}
The next example shows that for any set of vertices, we can ``blow up'' any group to get a groupoid homotopy equivalent to the group having the given vertex set. This will be useful later when we want to consider a group~$G$ together with a family of subgroups $(A_i)_{i\in I}$ as a pair of groupoids by considering $(G_I,\amalg_{i\in I} A_i)$.
\begin{def.}\label{d:blowup}
 Let $G$ be a group and $C$ a set. We define a groupoid $G_{C}$ by setting 
\begin{itemize}
 \item Objects: $\ob G_C:= C$.
 \item Morphisms: $\fa{e,f\in C} \Mor_{G_C}(e,f) := G$.
\end{itemize}
We then define composition as the multiplication of elements in $G$.
\end{def.}

\begin{def.}
 Let $f,g\colon (\G,\A)\longrightarrow (\HG,\B)$ be maps of pairs of groupoids.  We call a homotopy $h$ between $f$ and $g$ a \emph{homotopy relative \A} if for all $a\in \A$ the morphism $H_a\colon f(a)\longrightarrow g(a)$ is contained in $\B$. 

In this situation, we also write $f\simeq_{\A,\B} g$. The restriction of $h$ induces then a homotopy $h|_{\A}$ between $f|_\A, g|_\A\colon \A\longrightarrow \B$. 
\end{def.}
\subsection{The Fundamental Groupoid}

For us, the most important examples of groupoids, besides (families of) groups, will be given by fundamental groupoids of topological spaces. These are straightforward generalisations of fundamental groups:
\begin{def.}
 Let $X$ be a topological space and $I\subset X$ a subset. We define the groupoid $\pi_1(X,I)$ with object set $I$ by setting
\begin{align*}
 \fa{i,j\in I} \Mor_{\pi_1(X,I)}(i,j) = \{c\colon [0,1]\longrightarrow X\mid c \text{ a path from $i$ to $j$ in $X$}\}/ \sim,
\end{align*}
where $\sim$ denotes homotopy relative endpoints. We define composition via concatenation of paths. This is a well-defined groupoid, called \emph{the fundamental groupoid of~$X$ with respect to $I$}. We also write $\pi_1(X):= \pi_1(X,X)$.
\end{def.}
\begin{exa}
 If $X$ is a space and $x\in X$, then $\pi_1(X,\{x\})$ is just the fundamental group of $X$ with respect to the base point $x$.  
\end{exa}

\begin{def.}
 Let $(X,I)$ and $(Y,J)$ be pairs of topological spaces. A continuous map $f\colon(X,I) \longrightarrow (Y,J)$ induces in the obvious way a groupoid map $\pi_1(f) \colon \pi_1(X,I)\longrightarrow \pi_1(Y,J)$:
\begin{enumerate}
 \item On objects we define $\pi_1(f)$ via the map \begin{align*}
 f|_{I}\colon I &\longrightarrow J\\
         {i} &\longmapsto f(i).
       \end{align*}
\item On morphisms, we set for each $i,j\in I$ 
\begin{align*}
 \Mor_{\pi_1(X,I)}(i,j) &\longrightarrow \Mor_{\pi_1(Y,J)}(f(i),f(j))\\
 [\alpha] &\longmapsto [f\circ \alpha].
\end{align*}
\end{enumerate}
This defines a functor $\pi_1\colon \Top2 \longrightarrow \Grp$.
\end{def.}
\begin{prop}
 Let $(X,I)$ and $(Y,J)$ be pairs of topological spaces. Consider maps $f,g\colon (X,I)\longrightarrow (Y,J)$. If $f$ and $g$ are homotopic, so are $\pi_1(f)$ and $\pi_1(g)$.
\end{prop}
\begin{proof}
If $H\colon X\times [0,1]\longrightarrow Y$ is a homotopy between $f$ and $g$, then the family of homotopy classes $([H(i,\cdot)])_{i\in I}$ defines a natural equivalence between $\pi_1(f)$ and $\pi_1(g)$. 
\end{proof}

\section{Bounded Cohomology for Groupoids}

\subsection{Banach Modules over Groupoids} We now extend the notion of a Banach module over a group and several related algebraic concepts to the groupoid case.
\begin{def.}[Normed groupoid modules]Let $\G$ be a groupoid. 
\begin{enumerate}
 \item A \emph{normed (left) $\G$-module} $V=(V_e)_{e\in\ob\G}$ consists of:
\begin{enumerate}
 \item A family of normed real vector spaces $(V_e)_{e\in\ob\G}$.
\item An isometric action of $\G$ on $V$, i.e. for each $g\in\G$ an isometry
\begin{align*}
\rho_g\colon V_{s(g)}&\longrightarrow V_{t(g)}\\
 v&\longmapsto  \rho_g(v)=: gv,
\end{align*}
such that:
\begin{enumerate}
 \item For all $g,h\in \G$ with $s(g)=t(h)$ we have $\rho_g\circ \rho_h$ = $\rho_{gh}$. 
\item For all $i\in \ob \G$ we have $\rho_{\id_i} = \id_{V_i}$. 
\end{enumerate}
\end{enumerate}
\item A normed $\G$-module  $V$ is called a \emph{Banach $\G$-module} if in addition for each~$e\in\ob \G$ the normed $\R$-module $(V_e,\|\cdot\|)$ is a Banach space. 
\end{enumerate}
\end{def.}
\begin{def.}[\G-maps]Let $\G$ be a groupoid. 
\begin{enumerate}
 \item Let $V$ and $W$ be normed $\G$-modules. A bounded \emph{$\R$-morphism between $V$ and $W$} is a family~$(f_e\colon V_e\longrightarrow W_e)_{e\in\ob\G}$  of bounded $\R$-linear maps, such that the supremum $\|f\|_{\infty} := \sup_{e\in\ob\G} \|f_e\|_{\infty}$ exists.

\item Let $f\colon V\longrightarrow W$ be a bounded $\R$-morphism between $\G$-modules. We call $f$ \text{\emph{a~$\G$-map}} if for all $g\in\G$ we have $\rho_g^W\circ f_{s(g)} = f_{t(g)}\circ \rho^V_g$.
\end{enumerate}
\end{def.}
\begin{rem}\label{r:categorical}
 By definition, a normed left $\G$-module is nothing else than a covariant functor $\G\longrightarrow \normedmod$ that factors through $\normedmodone$, the category of normed $\R$-modules with norm non-increasing maps. 

A bounded $\G$-map  is then a uniformly bounded natural transformation between such functors. We will sometimes use this functorial description to define some concepts more concisely. 

Similarly, we also define \emph{normed $\G$-(co)chain complexes}.
\end{rem}
\begin{def.}
 Let $\G$ be a groupoid. We write $\normedgmod$ for the category of normed $\G$-modules together with bounded $\G$-maps. 
\end{def.}
\begin{def.}[The trivial $\G$-module]
 Let $\G$ be a groupoid. We endow the family  $(\R_e)_{e\in\ob\G} =(\R)_{e\in\ob\G}$ with the  trivial $\G$-action given by
\begin{align*}
 \rho_g =\id_\R \colon \R_{s(g)} &\longrightarrow \R_{t(g)}
\end{align*}
for all $g\in\G$. This is a Banach $\G$-module, which we denote by $\R_{\G}$.
\end{def.}

\begin{def.} Let $\A$ and $\G$ be groupoids and let  $f\colon \A\longrightarrow\G$ be a groupoid map.
\begin{enumerate}
 \item   Let $U\colon \G\longrightarrow \normedmod$ be a normed $\G$-module. We call the normed $\A$-module $f^\ast U:= U\circ f$ \emph{the induced $\A$-module structure on $U$}.
\item Let $\varphi \colon U\longrightarrow V$ be a $\G$-map. We define an $\A$-map 
\[ f^{\ast} \varphi \colon f^{\ast} U\longrightarrow f^{\ast} V,\]
by setting $((f^{\ast}\varphi)_e)_{e\in \ob \A} = (\varphi_{f(e)})_{e\in\ob\A}$. This defines a functor \[f^{\ast}\colon \normedgmod\longrightarrow \normedamod.\] 
\end{enumerate}
\end{def.}
\begin{lemma}\label{l:hinduced}
 Let $\G$ and $\HG$ be groupoids and let $V$ be an $\HG$-module. Let $f_0,f_1\colon \G\longrightarrow \HG$ be groupoid maps and $h$ a homotopy between $f_0$ and $f_1$. Then 
\begin{align*}
V\circ h\colon f^{\ast}_0 V&\longrightarrow f^{\ast}_1 V\\
V_{f_0(e)}\ni v &\longmapsto h_{e}\cdot v                                                                                                                                                                                                                                                                                                                                        \end{align*}
is a $\G$-isometry.  
\end{lemma}
\begin{proof}
 By Remark~\ref{r:categorical}, $V\circ h$ is a $\G$-map. Its inverse is given by $V\circ \overline h$, where~$\overline h$ denotes the inverse homotopy to $h$.
\end{proof}
\noindent We will also use a multiplicative notation and write $h\cdot v$ to denote $(V\circ h)(v)$. 
\begin{def.}
 Let $\G$ be a groupoid and and let $(V,\|\cdot\|_V)$ and $(W,\|\cdot\|_W)$ be normed \G-modules. Consider the family 
\begin{align*}
B(V,W) := (B(V_e,W_e))_{e\in\ob\G}
\end{align*}
 Here, $B(V_e,W_e)$ denotes the space of bounded linear functions from $V_e$ to~$W_e$, endowed with the family of operator norms $\|\cdot\|_{\infty}$ induced by the families of norms on $V$ and $W$. We define an isometric \G-action on $B(V,W)$ by setting 
for all $g\in \G$ and all $f\in \Hom(V_{s(g)},W_{s(g)})$ 
\begin{align*}
\fa{v\in V_{t(g)}} (g\cdot f)(v) =  g\cdot f(g^{-1}\cdot v).
\end{align*}
 If $W$ is a Banach $\G$-module, so is $B(V,W)$. This is functorial with respect to bounded $\G$-maps. 
\end{def.}
\newpage
\begin{rem}
Alternatively, we can view $B(V,W)$ as the composition 
\begin{center}
\begin{tikzpicture}
\matrix (m) [scale = 0.5,matrix of math nodes, row sep=2.5em,
column sep=4em, text height=2.5ex, text depth=0.25ex]
{\G&(\normedmod)^2&\normedmod,
\\ };
\path[->]
(m-1-1) edge node[auto] {$(V,\overline{W})$}(m-1-2)
(m-1-2) edge node[auto] {$B$}  (m-1-3)
;
\end{tikzpicture}
\end{center}
where $\overline{W}$ is the contravariant functor given by inverting morphisms in $\G$ and then applying $W$.
\end{rem}

\begin{def.}
Let $V=(V_e,\|\cdot\|_e)_{e\in\ob \G}$ be a normed $\G$-module.  We call the $\R$-submodule 
\begin{align*}
 V^{\G}:= \Bigl\{v\in \prod_{e\in\ob\G} V_e\;\Bigm |\; \fa{g\in\G} g\cdot v_{s(g)} = v_{t(g)},\ \sup_{e\in\ob\G} \|v_e\|_e <\infty\Bigr\},
\end{align*}
of $\prod_{e\in\ob\G}V_e$ endowed with the norm given by setting for each $v\in V^\G$ 
\begin{align*}
 \|v\|:=\sup_{e\in\ob\G} \|v_e\|_e
\end{align*}
\emph{the invariants of $V$}.
Alternatively, this can also be viewed as the canonical model for the limit of $V\colon \G\longrightarrow \normedmodone.$
\end{def.}
This is functorial with respect to bounded $\G$-maps:
\begin{prop}
 The invariants define a functor 
\[
 (\;\cdot\;)^\G \colon \normedgmod \longrightarrow \normedmod. 
\]
\end{prop}
One important example of invariants will be $B_{\G}(V,W):= B(V,W)^{\G}$. 
\subsection{Definition of Bounded Cohomology} We now give our definition of bounded cohomology for groupoids via Bar resolutions. 
\begin{def.}[The Bar resolution for groupoids]\label{d:barres} Let $\G$ be a groupoid.
\begin{enumerate}
 \item For each $n\in\N$ and each $e\in\ob\G$, we set
\[
 P_n(\G)_e  = \{(g_0,\dots,g_n)\in\G^{n+1}\mid \fa{i\in \{0,\dots,n-1\}} s(g_i) = t(g_{i+1}), t(g_0)=e\}.
\]
Equivalently, $P_n(\G)_e$ is the set of all ($n+1$)-paths in $\G$ ending in $e$.
\item For each $n\in\N$, define a normed $\G$-module $C_n(G)=(C_n(G)_e)_{e\in\ob\G}$ as follows. For all $e \in \ob \G$, we set
\[
 C_n(G)_e := \R\langle P_n(\G)_e\rangle,
\]
endowed with the $\ell^1$-norm with respect to the basis $P_n(G)_e$. We then define an isometric $\G$-action on $C_n(\G)$ by setting for all $g\in\G$
\begin{align*}
 \rho_g\colon C_n(\G)_{s(g)} &\longrightarrow C_n(\G)_{t(g)}\\
  (g_0,\dots,g_n) &\longmapsto (g\cdot g_0, \dots,g_n).
\end{align*}
\item 
For each $n\in \N$, we define  boundary maps
\begin{align*}
 \partial_n\colon C_n(\G) &\longrightarrow C_{n-1}(\G)\\
(g_0,\dots,g_n) &\longmapsto \sum_{i=0}^{n-1} (-1)^i (g_0,\dots,g_i\cdot g_{i+1},\dots, g_n)\\
&\phantom{\longmapsto} + (-1)^n\cdot (g_0,\dots, g_{n-1}).
\end{align*}
These maps are clearly $\G$-equivariant and the usual calculation shows that this does indeed define a $\G$-chain complex. By definition, these maps are bounded $\G$-maps and $\|\partial_n\|_{\infty} \leq n+1$ for all $n\in\N$, so~$(C_\ast(\G),\partial_\ast)$ is a normed $\G$-chain complex. 
\end{enumerate}
\end{def.}
This is a resolution in the following sense:
\begin{rem}\label{r:chaincontraction}
Let $\G$ be a groupoid. Consider the canonical augmentation
\begin{align*}
 \varepsilon \colon C_0(\G) &\longrightarrow \R_\G\\
g&\longmapsto t(g)\cdot 1.
\intertext{Then $(C_n(\G),\partial_n)_{n\in\N}$ together with $\varepsilon$ is a $\G$-resolution of $\R_\G$. An $\R$-chain contraction $s_\ast$ is given by the $\R$-morphisms
}
 s_{-1}\colon \R_\G&\longrightarrow C_0(\G)\\
e&\longmapsto \id_e
\intertext{and for all $n\in\N$}
s_n\colon C_n(\G) &\longrightarrow C_{n+1}(\G)\\
(g_0,\dots,g_n) &\longmapsto (\id_{t(g_0)},g_0,\dots,g_n).
\end{align*}
\end{rem}
To study amenable groupoids, it will be useful to have also a homogeneous resolution chain isometric to the inhomogeneous Bar resolution:

\begin{def.}[The homogeneous Bar resolution]\label{d:hombar} Let $\G$ be a groupoid.

\begin{enumerate}
 \item For each $n\in\N$ and each $e\in\ob\G$, we set
\[
 S_n(\G)_e  = \{(g_0,\dots,g_n)\in\G^{n+1}  \mid \fa{i\in\{0,\dots,n\}} t(g_i) = e\}.
\]
 \item For each $n\in\N$, we define a normed $\G$-module $L_n(\G)$ by setting for each $e\in \ob G$
\begin{align*}
 L_n(\G)_e = \R\langle S_n(\G)_e\rangle,
\end{align*}
endowed with the $\ell^1$-norm with respect to the basis $S_n(\G)_e$ and defining a $\G$-action by setting for each $g\in\G$
\begin{align*}
 \rho_g\colon L_n(\G)_{s(g)} &\longrightarrow L_n(\G)_{t(g)}\\
  (g_0,\dots,g_n) &\longmapsto (g\cdot g_0, \dots,g\cdot g_n).
\end{align*}
\item 
For each $n\in \N$, we define  boundary maps
\begin{align*}
 \partial_n\colon L_n(\G) &\longrightarrow L_{n-1}(\G)\\
(g_0,\dots,g_n) &\longmapsto \sum_{i=0}^{n} (-1)^i (g_0,\dots,\hat{g_i},\dots, g_n).
\end{align*}
Again, these maps are clearly $\G$-equivariant and the usual calculation shows that this does indeed define a $\G$-chain complex. By definition, these maps are bounded $\G$-maps and $\|\partial_n\|_{\infty} \leq n+1$ for all $n\in\N$, so~$(C_\ast(\G),\partial_\ast)$ is a normed $\G$-chain complex. 
\end{enumerate}
\end{def.}
\begin{rem}
 If $\G$ is a group, then $C_{\ast}(\G)$, respectively $L_{\ast}(\G)$, coincide with the usual definition of the (inhomogeneous, respectively homogeneous) real Bar resolution of the group $\G$.
\end{rem}
\begin{prop}\label{p:homogeneous}
 The maps 
\begin{align*}
C_n(\G)&\longrightarrow L_n(\G)\\
(g_0,\dots,g_n) &\longmapsto (g_0,g_0\cdot g_1,\dots, g_0\cdot\cdots g_n)
\intertext{and}
L_n(\G)&\longrightarrow C_n(\G)\\
(g_0,\dots,g_n) &\longmapsto (g_0,g_0^{-1}\cdot g_1,\dots, g_{n-1}^{-1}\cdot g_n)
\end{align*}
are mutually inverse $\G$-chain isometries. 
\end{prop}
\begin{proof}
 The maps are obviously mutually inverse $\G$-maps in each degree. They are chain maps by the same calculation as in the group case~\cite[Section VI.13]{HS97} and norm non-increasing (and hence isometric) because they map simplices to simplices. 
\end{proof}

\begin{def.}[Functoriality]
 Let $f\colon \A\longrightarrow\G$ be a groupoid map. Then \emph{the induced map}
\begin{align*}
\begin{pmatrix}
  \begin{aligned}
   C_n(f)\colon C_n(\A)&\longrightarrow f^{\ast}C_n(\G)\\
 (a_0,\dots,a_n)&\longmapsto (f(a_0),\dots,f(a_n))
  \end{aligned}
\end{pmatrix}_{n\in\N}
\end{align*}
is an $\A$-chain map with respect to the induced $\A$-structure on $C_\ast(\G)$. The map $C_n(f)$ is bounded for each $n\in\N$ and satisfies $\|C_n(f)\|_{\infty}\leq 1,$ for~$C_n(f)$ maps simplices to simplices. 
\end{def.}

In order to study bounded cohomology with coefficients, it will be useful to define a domain category that encompasses both groupoids and coefficient modules. We follow Kenneth Brown~\cite[Section III.8]{Br82} here:

\begin{def.}[Domain Categories for Bounded Cohomology]
We define a category $\GrpBanc$ by setting:
\begin{enumerate}
\item Objects in $\GrpBanc$ are pairs $(\G,V)$, where $\G$ is a groupoid and $V$ is a Banach $\G$-module.
\item A morphism $(\G,V)\longrightarrow (\HG,W)$ in $\GrpBanc$ is a pair $(f,\varphi)$, where $f\colon \G\longrightarrow \HG$ is a groupoid map and $\varphi \colon f^{\ast}W\longrightarrow V$ is a bounded $\G$-map. 
\item Composition is defined by setting for all composable pairs of morphisms $(f,\varphi)\colon (\G,V)\longrightarrow (\HG,W)$ and $(g,\psi)\colon (\HG,W)\longrightarrow (\A,U)$:
\[
 (g,\psi)\circ(f,\varphi) := (g\circ f, \varphi\circ (f^\ast \psi) ).
\]
\end{enumerate}
\end{def.}
\begin{def.}[The Banach Bar Complex with Coefficients]
 Let $\G$ be a groupoid and $V$ a Banach $\G$-module. 
\begin{enumerate}
\item We write 
\[
 C^{\ast}_b(\G;V):= B_{\G}(C_{\ast}(\G),V).
\]
Together with $\|\cdot\|_{\infty}$, this is a normed $\R$-cochain complex. 
\item If $(f,\varphi)\colon (\G,V)\longrightarrow (\HG,W)$ is a morphism in $\GrpBanc$, we write $ C^{\ast}_b(f;\varphi)$ for the map
\begin{align*}
 C^{\ast}_b(\HG;W)&\longrightarrow C^{\ast}_b(\G;V)\\
(w_e)_{f\in\ob\HG}&\longmapsto (\varphi \circ w_{f(e)}\circ C_{\ast}(f))_{e\in\ob \G}.
\end{align*}

This defines a contravariant functor $C_b^\ast \colon \GrpBanc \longrightarrow \RCh^{\|\cdot\|}$.
\end{enumerate}
\end{def.}
\begin{def.}
 Let $\G$ be a groupoid, $V$ a Banach $\G$-module.  We call the cohomology 
\[
 H^{\ast}_{b}(\G;V):= H^{\ast}(B_{\G}(C_{\ast}(\G),V)),
\]
together with the induced semi-norm on $H^{\ast}_{b}(\G;V)$, the \emph{bounded cohomology of $\G$ with coefficients in $V$}. 

This defines a contravariant functor $H^\ast_b\colon\GrpBanc \longrightarrow \Modgrn$.
\end{def.}

\begin{rem}\label{r:boundedell1conincide}
 As before, if $\G$ is a group, our definition of bounded cohomology coincides with the usual one.
\end{rem}
\subsection{Elementary Properties}
In this part, we show that bounded cohomology is an additive homotopy invariant for groupoids and see that it can be calculated by the  bounded cohomology of vertex groups.
\begin{prop}\label{p:bgrphomotopy}
 Let $\G$ and $\HG$ be groupoids. Let $f_0,f_1\colon \G\longrightarrow \HG$ be groupoid maps and let $h$ be a homotopy from $f_1$ to $f_0$. 
\begin{enumerate}
 \item For each $n\in\N$ and each $i\in\{0,\dots,n\}$ define a bounded $\G$-map
\begin{align*}
 s_n^i \colon C_n(\G) &\longrightarrow f_0^{\ast} C_{n+1}(\HG)\\
 (g_0,\dots,g_n) &\longmapsto (f_0(g_0),\dots,f_0(g_{i}),h_{s(g_{i})},f_1(g_{i+1}),\dots, f_1(g_{n})).
\end{align*}
Define for each $n\in\N$ a $\G$-map 
\begin{align*}
 s^h_n:= \sum_{i=0}^n (-1)^i s_n^i.
\end{align*}
Then $\partial_{n+1}\circ s^h_n + s^h_{n-1} \circ\partial_n = C_{n}(f_0)- h\cdot C_{n}(f_1)$ for each $n\in\N$, i.e.,~$s^h_{\ast}$ is a bounded $\G$-chain homotopy between $C_{\ast}(f_0)$ and $h\cdot C_{\ast}(f_1)$. 
\item Let $V$ be a Banach $\G$-module. The family
\begin{align*}
 s^\ast_{V,h} := (B_{\G}(s_{\ast}, \id_{f^{\ast}_0 V}))\circ j\colon C^\ast_b(\HG,V) &\longrightarrow C^{\ast-1}_b(\G;f_0^{\ast}V)\\
 (w_e)_{e\in\ob \HG} &\longmapsto (w_{f_0(e)}\circ s_{\ast,e})_{e\in\ob \G}. 
\end{align*}
is an $\R$-cochain homotopy between  $C_b^\ast(f_0;V)$ and~$C_b^{\ast}(f_1;V\circ h)$.
\item In particular,
\[
 H^\ast_b (f_0;V) = H^{\ast}_b(\id_{\G}; V\circ \overline{h})\circ H^{\ast}_b(f_1;V);
\]
and the map $H^{\ast}_b(\id_{\G}; V\circ \overline{h})$ is an isometric isomorphism. 
\end{enumerate}
\end{prop}
\begin{proof} A short calculation shows that $s^h_\ast$ is a bounded $\G$-chain homotopy.
It remains to show that $H^{\ast}_b(\id_{\G}; V\circ \overline{h})$ is isometric. But 
\begin{align*}
V\circ \overline{h}\colon f^{\ast}_0 V&\longrightarrow f^{\ast}_1 V\\
V_{f_0(e)}\ni v &\longmapsto \overline{h}_{e}\cdot v                                                                                                                                                                                                                                                                                                                                        \end{align*}
is isometric since the $\HG$-action on $V$ is isometric, hence all the induced maps are also isometric. 
\end{proof}

\begin{cor}\label{c:htpinvghn}
 Let $f\colon \G\longrightarrow \HG$ be an equivalence between groupoids and~$V$ a Banach $\HG$-module. Then the induced maps 
\[
 H^{\ast}_b(f;V)\colon H^{\ast}_b(\HG;V) \longrightarrow H^{\ast}_b(\G;f^{\ast}V)
\]
is an isometric isomorphism of semi-normed graded $\R$-modules.
\end{cor}

\begin{cor}[Bounded group cohomology calculates bounded groupoid cohomology]
Let $\G$ be a connected groupoid and $V$ a Banach $\G$-module. Let $e\in\ob \G$ be a vertex and $i_e\colon \G_e\longrightarrow \G$ the inclusion of the corresponding vertex group. Then 
\[
 H^{\ast}_b(i_e,V) \colon H^{\ast}_b(\G,V) \longrightarrow H^{\ast}_b(\G_e;i_e^{\ast}V)
\]
is an isometric isomorphism of normed graded $\R$-modules. 
\end{cor}
\begin{proof}
 By Corollary~\ref{c:grphequiv}, $i_e$ is a homotopy equivalence and by Corollary~\ref{c:htpinvghn}, the induced maps are isometric isomorphisms. Thus the result follows form Remark~\ref{r:boundedell1conincide}.
\end{proof}

\begin{def.}\label{d:prodnorm}
 Let $(V_i,\|\cdot\|_i)_{i\in I}$ be a family of (semi-)normed $\R$-modules. We define \emph{the normed product} of $(V_i,\|\cdot\|_i)_{i\in I}$ as the $\R$-module
\[
\prod_{i\in I}^{\|\cdot\|} (V_i,\|\cdot\|_i) := \Bigl\{ v\in \prod_{i\in i} V_i\,\Bigm|\, \sup_{i\in I}\|v_i\|<\infty\Bigr\}
\]
endowed with the product norm 
\begin{align*}
\prod_{i\in I}^{\|\cdot\|} (V_i,\|\cdot\|_i) &\longrightarrow \R\\
  (v_i)_{i\in I} &\longmapsto \sup \{\|v_i\|_i\mid i\in I\}.
\end{align*}
If $I$ is finite, the normed product of $(V_i,\|\cdot\|_i)_{i\in I}$ coincides with $\prod_{i\in I} V_i$ endowed with the maximum norm.  
\end{def.}

\begin{prop}[Bounded groupoid cohomology and disjoint unions]\label{p:grpadditive}
 Let~$\G$ be a groupoid  and $V$ a Banach~$\G$-module. Let $\G= \amalg_{\lambda\in\Lambda} \G^\lambda$ be the partition of $\G$ into connected components. For each~$\lambda\in\Lambda$, write $V^\lambda$ for the $\G$-module structure on $V$ induced by the inclusion $\G^{\lambda}\longhookrightarrow \G$. Then the family of inclusions $(\G^{\lambda}\longhookrightarrow \G)_{\lambda\in\Lambda}$ induces an isometric isomorphism
\begin{align*}
 H^{\ast}_b(\G;V) &\longrightarrow \prod_{\lambda\in\Lambda}^{\|\cdot\|} H^{\ast}_b(\G^{\lambda};V^{\lambda})
\end{align*}
\end{prop}

\begin{proof}
  We see directly that the splitting of $C_{\ast}(\G)$ with respect to the connected components is preserved after applying $B_{\G}(\;\cdot\;, V)$ and that the norm is exactly the the product norm. 
\end{proof}
\subsection{Relative Bounded Cohomology}
We now define bounded cohomology for pairs of groupoids. Using this definition, we derive a long exact sequence for bounded cohomology of groupoids. Then, we show that relative bounded cohomology is a homotopy invariant. Finally, we discuss the special case of a group relative to a family of subgroups and get a long exact sequence relating bounded cohomology of a group relative to a  family of subgroups to the regular bounded homology of the group and the subgroups.  
\begin{def.}
 Let $i\colon \A \longhookrightarrow \G$ be a groupoid pair, i.e. the canonical inclusion of a subgroupoid $\A\subset \G$. Let $V$ be a Banach $\G$-module.
\begin{enumerate}
\item The map 
\begin{align*}
 C^{\ast}_b(i;V)\colon C^{\ast}_b(\G;V) &\longrightarrow C^{\ast}_b(\A,i^{\ast} V)\\
 (f_e)_{e\in\ob\G}&\longmapsto (f_e|_{C_{\ast}(\A)_e})_{e\in\ob\A}
\end{align*}
is surjective. Its kernel
\[
 C^{\ast}_b(\G,\A;V):= \{f\in B_{\G}(C_{\ast}(\G),V) \mid \fa{e\in\ob \A} f_e|_{C_{\ast}(\A)_e} =0\},
\]
endowed with the induced norm on the subspace, is a normed $\R$-cochain complex. We write $\iota^{\ast}_{(\G,\A;V)}\colon C^{\ast}_b(\G,\A;V) \longrightarrow C^{\ast}_b(\G;V)$ for the canonical inclusion.
\item We call 
\[
 H^{\ast}_b(\G,\A;V):= H^{\ast}(C^{\ast}_b(\G,\A;V))
\]
the \emph{bounded cohomology of $\G$ relative to $\A$ with coefficients in $V$}.
\end{enumerate}
\end{def.}
We define domain categories $\GrppBan$ and $\GrppBanc$ analogously to $\GrpBan$ and $\GrpBanc$ by considering pairs of groupoids instead of groupoids. 
\begin{def.}
 Let $(f,\varphi)\colon ((\G,\A),V)\longrightarrow ((\HG,\B),W)$ be a morphism in $\GrppBanc$. Then $C^{\ast}_b(f,\varphi)$ restricts to a bounded cochain map 
\begin{align*}
 C^{\ast}_{b}(f,f|_{\A};\varphi):= C^{\ast}_b(f;\varphi)|_{C^{\ast}_b(\HG,\B;W)}\colon C^{\ast}_{b}(\HG,\B;W)\longrightarrow C^{\ast}_{b}(\G,\A;V).
\end{align*}
This induces a continuous map between graded normed modules
\begin{align*}
  H^{\ast}_{b}(f,f|_{\A};\varphi)\colon H^{\ast}_{b}(\HG,\B;W)\longrightarrow H^{\ast}_{b}(\G,\A;V).
\end{align*}
This defines a contravariant functor $\GrppBanc\longrightarrow \Modgrn$.
\end{def.}

As for the above definition, we just note that the right-hand square in the following diagram commutes by definition, so the map on the left-hand side is defined:
\begin{center}
\begin{tikzpicture}
\matrix (m) [matrix of math nodes, row sep=3.5em,
column sep=3.5em, text height=1.5ex, text depth=0.25ex]
{0&C_b^{\ast}(\HG,\B;W)&C^{\ast}_b(\HG;W) & C^{\ast}_b(\B;i_{\B}^{\ast}W)&0\\
0&C_b^{\ast}(\G,\A;V)& C^{\ast}_b(\G;V)& C^{\ast}_b(\A;i_{A}^{\ast}V)&0
\\ };
\path[->]
(m-1-1) edge (m-1-2)
(m-1-2) edge node[auto] {$\iota_{(\HG,\B;W)}^{\ast}$} (m-1-3)
(m-1-4) edge (m-1-5)
(m-2-1) edge (m-2-2)
(m-2-2)  edge node[below] {$\iota_{(\G,\A;V)}^{\ast}$}(m-2-3)
(m-2-4) edge (m-2-5)
(m-1-2) edge node[auto] {$ C^{\ast}_b(f,f|_{\A};\varphi)$}  (m-2-2)
(m-1-3) edge node[auto] {$ C^{\ast}_b(i_\B;V)$}(m-1-4)
(m-1-3) edge node[auto] {$ C^{\ast}_b(f;\varphi)$}  (m-2-3)
(m-2-3) edge node[below] {$C^{\ast}_b(i_\A;V)$}  (m-2-4)
(m-1-4) edge node[auto] {$ C^{\ast}_b(f|_{\A};i^{\ast}_\A \varphi) $}  (m-2-4)
;
\end{tikzpicture}
\end{center}
Here $i_{\A}^{\ast}\varphi \colon f|_{\A}^{\ast} i_{\B}^{\ast} W = i_{\A}^{\ast} f^{\ast} V\longrightarrow i_{\A}^{\ast} V$ denotes the $\A$-map  $(\varphi_a)_{a\in\ob\A}$, i.e., by restricting the family corresponding to $\varphi$ to objects in $\A$. Functoriality then follows from the functoriality of the right-hand square. 

\begin{prop} Let $i_{\A}\colon \A\longrightarrow \G$ be the inclusion of a subgroupoid .There is a natural (with respect to morphisms in $\GrppBanc$) long exact sequence
\makebox[12.5cm]{%
\begin{centering}
\begin{tikzpicture}[ampersand replacement=\&]
\matrix (m) [scale = 0.5,matrix of math nodes, row sep=1.5em,
column sep=1.5em, text height=1.5ex, text depth=0.25ex]
{\cdots\&H_b^\ast(\G,\A;V)\&H_b^\ast(\G;V)\&H^\ast_b(\A;i_{\A}^{\ast}V)\&H_b^{\ast+1}(\G,\A;V)\&\cdots
\\ };
\path[->]
(m-1-1) edge node[auto] {}(m-1-2)
(m-1-2) edge node[above = 3pt] {$ H^{\ast}(\iota_{(\G,\A;V)}^\ast) $} (m-1-3)
(m-1-3) edge node[above = 3pt] {$ H^\ast_b(i_{\A};V) $}   (m-1-4)
(m-1-4) edge node[above = 3pt] {$  \delta^\ast$}  (m-1-5)
(m-1-5) edge node[auto] {$  $}  (m-1-6)
;
\end{tikzpicture}
\end{centering}
}

\noindent
such that $\delta^\ast$ is continuous with respect to the induced semi-norms. 
\end{prop}
\begin{proof}
 Take care that the boundary map is continuous in the usual diagram chase in the proof of the snake lemma. 
\end{proof}

\begin{prop}\label{p:relhomotopy} \hfill
\begin{enumerate}
\item Let $f,g\colon (\G,\A)\longrightarrow (\HG,\B)$ be maps of pairs of groupoids and $V$ be an~$\HG$-module. If $f\simeq_{\A,\B} g$ via a homotopy $h$, then there is a canonical $\R$-cochain homotopy between $C^{\ast}_b(f,f|_{\A};V)$ and $C^{\ast}_b(g,g|_{\A};V\circ \overline h)$.
\item In particular,
\[
 H^\ast_b (f,f|_{\A};V) = H^{\ast}_b(\id_{\G},\id_{\A}; V\circ \overline{h})\circ H^{\ast}_b(g,g|_{\A};V);
\]
and the map $H^{\ast}_b(\id_{\G},\id_{\A}; V\circ \overline{h})$ is an isometric isomorphism. 
\end{enumerate}
\end{prop}
\begin{proof}\hfill
 \begin{enumerate}
  \item  Let $s^{\ast}_{V,h}$ and $ s^{\ast}_{i_{\B}^{\ast}V,h|_{\A}}$ denote the cochain homotopies induced by $h$ and $h|_{\A}$ constructed in Proposition~\ref{p:bgrphomotopy}.  The right-hand side square  of the following diagram commutes by definition:
\begin{center}
\begin{tikzpicture}
\matrix (m) [matrix of math nodes, row sep=3.5em,
column sep=1.em, text height=1.5ex, text depth=0.25ex]
{0&C_b^{\ast}(\HG,\B;V)&&C^{\ast}_b(\HG;V) && C^{\ast}_b(\B;i_{\B}^{\ast}V)&0\\
0&C_b^{\ast-1}(\G,\A;f^{\ast}V)&& C^{\ast-1}_b(\G;f^{\ast}V)&& C^{\ast-1}_b(\A;i_{A}^{\ast}f^{\ast}V)&0.
\\ };
\path[->]
(m-1-1) edge (m-1-2)
(m-1-2) edge node[auto] {$\iota_{(\HG,\B;V)}^{\ast}$} (m-1-4)
(m-1-6) edge (m-1-7)
(m-2-1) edge (m-2-2)
(m-2-2)  edge node[below] {$\iota_{(\G,\A;f^{\ast}V)}^{\ast-1}$}(m-2-4)
(m-2-6) edge (m-2-7)
(m-1-2) edge node[auto] {$ s^{\ast}_{V,h}|_{C^{\ast}(\HG,\B;V)}$}  (m-2-2)
(m-1-4) edge node[auto] {$ C^{\ast}_b(i_\B;V)$}(m-1-6)
(m-1-6) edge node[auto] {$ s^{\ast}_{i_{\B}^{\ast}V,h|_{\A}}$}  (m-2-6)
(m-2-4) edge node[below] {$C^{\ast-1}_b(i_\A;f^{\ast}V)$}  (m-2-6)
(m-1-4) edge node[auto] {$ s^{\ast}_{V,h} $}  (m-2-4)
;
\end{tikzpicture}
\end{center}
Hence the map on the left-hand side is defined.  The map $s^{\ast}_{i_{\B}^{\ast}V,h|_{\A}}$ is a cochain homotopy between  $C^{\ast}_b(f|_{\A},i_{\B}^{\ast}V)= C^{\ast}_b(f|_{\A}, i_{\A}^{\ast} \id_{f^{\ast} V})$ and~$C^{\ast}_b(g|_{\A}, (V\circ\overline{h|_{\A}}))= C^{\ast}_b(g|_{\A}, i_{\A}^{\ast} (V\circ\overline{h})).$
Hence, by comparing the diagrams defining $C^{\ast}_b(f,f|_{\A};V)$, $C^{\ast}_b(g,g|_{\A};V)$ and $s^{\ast}_{V,h}|_{C^{\ast}(\HG,\B;V)}$, we see that $s^{\ast}_{V,h}|_{C^{\ast}(\HG,\B;V)}$ defines an $\R$-cochain homotopy between $C^{\ast}_b(f,f|_{\A};V)$ and $C^{\ast}_b(g,g|_{\A};V)$.

\item The map $C^{\ast}_b(\id_{\G},\id_{\A}; V\circ \overline{h})$ is an isometric cochain isomorphism, since $C^{\ast}_b(\id_{\G}; V\circ \overline{h})$ and $C^{\ast}_b(\id_{\A}; i_{A}^{\ast}(V\circ \overline{h}))=C^{\ast}_b(\id_{\A}; i_{\A}^{\ast}V\circ \overline{h|_{\A}}))$ are isometric cochain isomorphisms. \qedhere
\end{enumerate}

\end{proof}
\begin{cor}\label{c:homotopyinvrelcon}
 Let $f\colon (\G,\A)\longrightarrow (\HG,\B)$ be an equivalence relative $(\A,\B)$, i.e., there exists a map $g\colon (\HG,\B)\longrightarrow (\G,\A)$, such that $g\circ f\simeq_{\A,\B} \id_G$ and~$f\circ g\simeq_{\B,\A} \id_\HG$. Then \[H^{\ast}_b(f,f|_{\A};V)\colon H^\ast_b(\HG,\B;V)\longrightarrow H_b^\ast(\G,\A;f^{\ast} V)\] is an isometric isomorphism.
\end{cor}
\begin{cor}
Let $(\G,\A)$ be a pair of connected groupoids with vertex group~$G$ and $A$ respectively. Then we get an isometric isomorphism
\[
 H_b^\ast(\G,\A;\R_\G) \cong H^\ast_b(G,A;\R).
\]
\end{cor}
\begin{proof}
 As in the absolute case, it is straightforward to see that $(\G,\A)$ and~$(G,A)$ are equivalent relative ($\A,A)$.
\end{proof}
\begin{def.}
 Let $G$ be a group and $(A_i)_{i\in I}$ be a family of subgroups. Let $V$ be a Banach $G$-module. We define
\[
H^{\ast}_b(G,(A_i)_{i\in I};V) :=  H^{\ast}_b(G_I,\amalg_{i\in I} A_i;V_I)
\]
This is functorial in the obvious way.
\end{def.}
\begin{lemma}\label{l:htpinclusion}
 Let $G$ be a group and $I$ a set. For each $i\in I$ let $l_i\colon G\longrightarrow G_I$ denote the canonical inclusion as a vertex group. Then 
\[
 H_b^{\ast}(l_i;V_I) = H_b^{\ast}(l_j;V_I)
\]
for all $i,j\in I$ and for all Banach $G$-modules $V$. 
\end{lemma}
\begin{proof}
 We can define a homotopy between $l_i$ and $l_j$ by setting \[h_1:= \id_G\colon G= G_i \longrightarrow G_j = G.\] Here 1 denotes the unique vertex of $G$. Hence, by Proposition~\ref{p:bgrphomotopy}
\[
 H^{\ast}_b(l_i;V_I) = H_b^{\ast}(\id_G; V_{I}\circ h)\circ H_b^{\ast}(l_j;V_I) = H_b^{\ast}(l_j;V_I).\qedhere
\]
\end{proof}

\begin{thm}
 Let $G$ be a group and $(A_i)_{i\in I}$ a family of subgroups. Let $V$ be a Banach $G$-module. Then there is a natural long exact sequence

\

\makebox[\textwidth]{%
\begin{centering}
\begin{tikzpicture}[ampersand replacement=\&]
\matrix (m) [scale = 0.5,matrix of math nodes, row sep=1.5em,
column sep=1.em, text height=1.5ex, text depth=0.25ex]
{\cdots\&H_b^\ast(G,(A_i)_{i\in I};V)\&H_b^\ast(G;V)\&\prod_{i\in I}^{\|\cdot\|}H^\ast_b(A_i;i_{A_i}^{\ast}V)\&H_b^{\ast+1}(G,(A_i)_{i\in I};V)\&\cdots
\\ };
\path[->]
(m-1-1) edge node[auto] {}(m-1-2)
(m-1-2) edge node[above = 8pt] {$  \iota^{\ast}$} (m-1-3)
(m-1-3) edge node[above = 8pt] {$  (H_{b}^{\ast}(i_{A_i}; V))_{i\in I}
$}   (m-1-4)
(m-1-4) edge node[above = 8pt] {$  d^\ast$}  (m-1-5)
(m-1-5) edge node[auto] {$  $}  (m-1-6)
;
\end{tikzpicture}
\end{centering}
}

\noindent
such that $d^\ast$ is continuous with respect to the induced semi-norm and the product semi-norm. Here 
\[
 \iota^{\ast}:= H_{b}^{\ast}(l_e;V_I)\circ H^{\ast}(\iota_{G_I,A_I,V_I}).
\]
where $l_e\colon G\longrightarrow G_I$ is the canonical inclusions for some vertex $e\in\ob\G$.  

\end{thm}
\begin{proof}
Write $t_j\colon A_j\longhookrightarrow \amalg_{i\in I} A_i$ and $s\colon \amalg_{i\in I} A_i\longhookrightarrow G_I$ for the canonical inclusions. By Corollary~\ref{c:grphequiv} and Proposition~\ref{p:grpadditive}, the rows in the following diagram are isometric isomorphisms:\

\

\makebox[12cm]{%
 \begin{centering}
\begin{tikzpicture}[ampersand replacement=\&]
\matrix (m) [scale = 0.3,matrix of math nodes, row sep=3.5em,
column sep=1.5em, text height=1.5ex, text depth=0.25ex]
{
H_{b}^{\ast}(G;(A_i)_{i\in I};V) \&H_b^{\ast}(G_I; V_I)\& H_b^{\ast}(\amalg_{i\in I} A_i;s^{\ast} V)\&H_b^{\ast+1}(G,(A_i)_{i\in I};V)\\
H_{b}^{\ast}(G;(A_i)_{i\in I};V)\& H_b^{\ast}(G;V)\&\prod_{i\in I} H_b^{\ast}(A_i; i_{A_i}^{\ast} V)\&H_b^{\ast+1}(G,(A_i)_{i\in I};V)\\
};
\path[->]
(m-1-1) edge node[auto] {$ = $} (m-2-1)
(m-1-2) edge node[auto] {$ H_b^{\ast}(l_e;V_I)$} node[left] {$\cong$} (m-2-2)
(m-1-3) edge node[right = 3pt] {$ (H_b^{\ast}(t_i;s^{\ast} V))_{i\in I}$} (m-2-3)
(m-1-4) edge node[right] {$ =$} (m-2-4)
(m-1-1) edge node[above=8pt] {$H^{\ast}(\iota^{\ast}_{(G,(A_i)_{i\in I};V)})$} (m-1-2)
(m-1-2) edge node[above=8pt] {$ H_b^{\ast}(s;V_I)$} (m-1-3)
(m-1-3) edge node[above=8pt] {$ \delta^{\ast}$} (m-1-4)
(m-2-1) edge node[below=8pt] {$ \iota^{\ast}$} (m-2-2)
(m-2-2) edge node[below=8pt] {$(H_b^{\ast}(i_{A_i};V))_{i\in I}$} (m-2-3)
(m-2-3) edge node[below=8pt] {$ d^{\ast}$} (m-2-4)
;
\end{tikzpicture}
\end{centering}
}
\

The upper square commutes by definition and we can choose a continuous map~$d^{\ast}$ in such a way, that the lower square commutes. The centre square commutes by Lemma~\ref{l:htpinclusion}. 
\end{proof}
\section{Relative Homological Algebra}
\subsection{Relative Homological Algebra for Groupoids}
In this section, we develop the relative homological algebra necessary to study resolutions that can calculate bounded cohomology of grou\-poids, analogously to the group case~\cite{I,Mo01}. We introduce the notations of relatively injective groupoid modules and define strong resolutions for Banach groupoid modules. Then, we show that $B(C_\ast(\G),V)$ is a strong relatively injective resolution for each Banach~$\G$-module $V$. Next, we prove the fundamental lemma for relative homological algebra in our setting, implying in particular that strong relatively injective resolutions are unique up to bounded $\G$-cochain equivalence. Thus these resolutions can be used to calculate bounded cohomology up to isomorphism, and we show that the semi-norm on bounded cohomology can be seen to be the infimum over all semi-norms induced by strong relatively injective resolutions. 
\begin{def.}\label{d:relinj}
 Let $\G$ be a groupoid.
\begin{enumerate}
 \item Let $V$ and $W$ be Banach \G-modules. A $\G$-map $i\colon V\longrightarrow W$ is called \emph{relatively injective} if there exists a  (not necessarily $\G$-equivariant) $\R$-morphism $\sigma\colon W\longrightarrow V$ such that $\sigma\circ i= \id_V$ and~$\|\sigma\|_{\infty} \leq 1$. 
\item A $\G$-module $I$ is called \emph{relatively injective} if for each relatively injective \G-map $i\colon V\longrightarrow W$ between Banach $\G$-modules and each $\G$-map $\alpha\colon V\longrightarrow I$ there is a $\G$-map $\beta\colon W\longrightarrow I$, such that $ \beta\circ i = \alpha$ and~$\|\beta\|_{\infty}\leq \|\alpha\|_{\infty}.$
\end{enumerate}
\end{def.}

\begin{prop}\label{p:boundedrelinj}
 Let $U$ be a Banach $\G$-module and $n\in\N$. Then the Banach $\G$-module $B(C_n(\G),U)$ is relatively injective.
\end{prop}
\begin{proof}
 The proof is basically the same as in the case of $\G$ being a group~\cite[Proposition 3.4.3]{Bl14}.\qedhere
\end{proof}
\begin{def.}[Strong Resolutions] Let $(C^\ast,\delta^{\ast})$ be a Banach $\G$-cochain complex, $V$ a Banach $\G$-module and $\varepsilon\colon V\longrightarrow C_0$ a $\G$-augmentation map. We call $(C^\ast,\delta^\ast,\varepsilon)$ \emph{strong} or \emph{a strong resolution for $V$} if there exists a norm non-increasing cochain contraction, i.e., a family 
\begin{align*}
(s^n\colon C^n &\longrightarrow C^{n-1})_{n\in \N_{>0}}\\
s^{0} \colon C_0 &\longrightarrow V
\end{align*}
of (not necessarily $\G$-equivariant) $\R$-morphisms between $\G$-modules such that for all $n\in\N$ we have $\|s^n\|_{\infty} \leq 1$ and the family $(s^n)_{n\in\N}$ is a cochain contraction of the augmented cochain complex.
\end{def.}

\begin{exa}\label{e:cochaincontr}
Consider the Banach $\G$-cochain complex $B(C_{\ast}(\G),V)$ together with the augmentation map $\varepsilon_{\G} \colon B(C_0(\G),V) \longrightarrow V$  given by
\begin{align*}
 \varepsilon_{\G} \colon V&\longrightarrow  B(C_0(\G),V) \\
 v &\longmapsto (g\longmapsto v).
\end{align*}
Then, the family $s_{\ast}$ defined in Remark~\ref{r:chaincontraction} induces a norm-non increasing $\R$-cochain contraction~$s_b^{\ast}$ of $(B(C_{\ast}(\G),V),\varepsilon_{\G})$ given by
\begin{align*}
s_{b}^{\ast}\colon  B(C_{\ast}(\G),V) &\longrightarrow B(C_{\ast-1}(\G),V)\\
 \varphi &\longmapsto \varphi\circ s_{\ast}
\intertext{and setting for each $e\in \ob\G$}
s_{b,e}^{0}\colon B(C_{0}(\G)_e,V_e) &\longrightarrow V_e\\
\varphi &\longmapsto \varphi (\id_e).
\end{align*}
In particular,  $(B(C_{\ast}(\G),V),\varepsilon_{\G})$ is a strong resolution of $V$. 
\end{exa}

\begin{prop}[Fundamental lemma]\label{t:fundamental}
 Let $(I^n,\delta^n_I)_{n\in\N}$ be a relatively injective Banach \G-cochain complex and $\varepsilon \colon W\longrightarrow I^0$ a \G-augmentation map. Let $((C^n,\delta^n_C)_{n\in\N}, \nu \colon V\longrightarrow C^0)$ be a strong \G-resolution of a Banach~\G-module~$V$. Let $f\colon V\longrightarrow W$ be a $\G$-morphism. Then there exists an up to bounded $\G$-cochain homotopy unique extension of $f$ to a bounded~$\G$-cochain map between the resolution $(C^n,\delta_C^n,\nu)$ and the augmented cochain complex~ $(I^n,\delta^n_I,\varepsilon)$. 
\end{prop}
\begin{proof} The proof is basically the same lifting argument as for all homological ``fundamental lemmas'', for instance~\cite[Proposition 3.4.7]{Bl14}.
\end{proof}

\begin{cor}
 Let $\G$ be a groupoid and $V$ a $\G$-module. Then there exists an up to canonical bounded $\G$-cochain homotopy equivalence unique strong relatively injective \G-resolution of $V$.
\end{cor}
\begin{prop}\label{p:normnonincr}
Let $\G$ be a groupoid and let $V$ be a Banach $\G$-module.  Let~$((D^\ast,\delta^\ast_D), \varepsilon \colon V \longrightarrow D^0)$ be a strong \G-resolution of $V$.

Then for each strong cochain contraction of $(D^{\ast},\varepsilon)$ there exists a canonical  norm non-increasing cochain map of this resolution to the standard resolution~$(B(C_n(\G),V))_{n\in\N}$ of $V$ extending $\id_V$. 
\end{prop}
\begin{proof}
 The proof is similar to the group case~\cite[Lemma~3.2.2]{I}.

 Let $((s^n\colon D^n\longrightarrow D^{n-1})_{n\in\N}, s^0\colon D^0\longrightarrow V)$ be a norm non-increasing cochain contraction of the augmented cochain complex.  We will define families~$(\alpha_e^n\colon D_e^n \longrightarrow B(C_n(\G)_e,V_e))_{e\in\ob\G}$ by induction over $n\in\N\cup\{-1\}$. First, we set $\alpha^{-1} = \id_{V}$. Assume we have defined $\alpha_{n-1}$ for some $n\in\N$. Then we set for all $(g_0,\dots,g_n)\in P_n(\G)$ and all $\varphi \in D^n_{t(g_0)}$   
\begin{align*}
\alpha^n_{t(g_0)}(\varphi)(g_0,\dots,g_n) = \alpha^{n-1}_{t(g_0)}(g_0\cdot s^{n}_{s(g_0)}(g_0^{-1}\cdot \varphi))(g_0\cdot g_1,\dots,g_n).
\end{align*}
We immediately see that this is a $\G$-map for all $n\in\N$ and since $s^n$ is norm non-increasing and the $\G$-action is isometric, $\alpha^n$ is norm non-increasing by induction. By a short calculation we see that $(\alpha^n)_{n\in\N}$ is a cochain map. 
\end{proof}
\begin{cor}\label{c:strongrelin}
 Let $\G$ be a groupoid and $V$ a Banach $\G$-module. Moreover, let~$((D^n,\delta^n_D)_{n\in\N}, \varepsilon \colon V\longrightarrow D^0)$ be a strong relatively injective~\G-resolution of $V$. Then there exists a canonical semi-norm non-increasing isomorphism of graded $\R$-modules
\[
 H^\ast (D^{\ast\G}) \longrightarrow H^\ast_b(\G;V).
\]
\end{cor}
\begin{proof}
 By Proposition~\ref{p:normnonincr} there exists a norm non-increasing $\G$-cochain map~$D^n \longrightarrow B(C_n(\G),V)$
extending $\id_{V}$. This map induces a norm non-increasing morphism $H^\ast(D^{\ast\G})\longrightarrow H_b^\ast(\G;V).$
By the fundamental lemma for groupoids, this is an isomorphism.
\end{proof}

\subsection{Relative Homological Algebra for Pairs of Groupoids}\label{ss:relhomalgpairs}
In this section, we will discuss a version of relative homological algebra that can be used to describe the bounded cohomology of a pair of groupoids. The definitions and results will be analogous to the ones in the absolute setting. If~$(\G,\A)$ is a pair of groupoids, we will define $(\G,\A)$-cochain complexes and strong, relatively injective $(\G,\A)$-resolutions in this setting. We will see that there is a fundamental lemma for pairs and that the pair of standard resolutions is a strong, relatively injective $(\G,\A)$-resolution, thus $H_b^\ast(\G,\A;V)$ can be calculated by such resolutions. Our definition is slightly more restrictive than Park's definition of allowable pairs~\cite{Pa03}, but will directly lead to the fundamental lemma for pairs. 
\begin{def.}
 Let $i\colon \A\longrightarrow \G$ be a pair of groupoids, i.e., $\A$ and $\G$ are groupoids and $i$ is an injective groupoid map, see Definition~\ref{d:groupoids}. We define a category $\GABan$ by setting:

\begin{enumerate}
 \item Objects in $\GABan$ are tripels $(V,V',\varphi)$, where $V$ is a Banach $\G$-module,~$V'$ a Banach $\A$-module and $\varphi \colon i^\ast V\longrightarrow V'$ an $\A$-morphism. We call such an object a \emph{$(\G,\A)$-module}.
\item A morphism $(j,j') \colon (V,V',\varphi)\longrightarrow (W,W',\psi)$ in $\GABan$ is a pair~$(j,j')$ where  $j\colon V\longrightarrow W$ is a $\G$-map and \text{$j'\colon V'\longrightarrow W'$} is an $\A$-map, such that the following diagram commutes:
\begin{center}
\begin{tikzpicture}
\matrix (m) [scale = 0.5,matrix of math nodes, row sep=2.5em,
column sep=4em, text height=2.5ex, text depth=0.25ex]
{i^{\ast} V^{\ast} &i^\ast W\\
V'& W' 
\\ };
\path[->]
(m-1-1) edge node[auto] {$ i^\ast j$}(m-1-2)
(m-1-1) edge node[auto] {$\varphi$}  (m-2-1)
(m-2-1) edge node[auto] {$j'$}(m-2-2)
(m-1-2) edge node[auto] {$\psi$}(m-2-2)
;
\end{tikzpicture}
\end{center}
Composition is then defined componentwise. We call such a morphism also a \emph{$(\G,\A)$-map}. In addition, we will consider not necessarily $(\G,\A)$-equivariant morphisms  $(j,j') \colon (V,V',\varphi)\longrightarrow (W,W',\psi)$ by dropping the condition that $j$ and $j'$ are equivariant, but still demanding that the above diagram commutes.
\item Similarly, we also define a category $\GACh$ of Banach $(\G,\A)$-cochain complexes. The notions of augmentations, cochain homotopies etc. translate naturally into this setting. 
\end{enumerate}
\end{def.}
Now we define cohomoloyg of $(\G,\A)$-cochain complexes:
\begin{def.} Let $i\colon \A\longrightarrow \G$ be a pair of groupoids.
\begin{enumerate}
 \item Let $(C^\ast,D^\ast,f^\ast)$ be a $(\G,\A)$-cochain complex. The map $f^\ast$ restricts to a cochain map 
\begin{align*}
 \overline{f^\ast}\colon C^{\ast\G} &\longrightarrow D^{\ast\A}\\
(v_e)_{e\in\ob \G} &\longmapsto (f^\ast_a(v_{i(a)}))_{a\in\ob \A}.
\end{align*}
We write $K^\ast(C^\ast,D^\ast,f^\ast)$ for the normed $\R$-cochain complex $\ker(\overline {f^\ast})$, given by considering the kernel in each degree and endowed with the norm induced by the norm on $C^\ast$. 
\item Let $(j,j')\colon (C_0^\ast,D_0^\ast,f_0^\ast)\longrightarrow (C_1^\ast,D_1^\ast,f_1^\ast)$ be a $(\G,\A)$-cochain map. Then, by restriction, $(j,j')$ induce the maps on the right-hand side of the following diagram
\begin{center}
\begin{tikzpicture}
\matrix (m) [scale = 0.5,matrix of math nodes, row sep=2.5em,
column sep=4em, text height=2.5ex, text depth=0.25ex]
{\ker \overline {f_0^\ast}&C_0^{\ast\G}&D_0^{\ast\A}\\
\ker \overline {f_1^\ast}&C_1^{\ast\G}&D_1^{\ast\A},
\\ };
\path[->]
(m-1-1) edge node[auto] {$ $}(m-1-2)
(m-1-1) edge node[auto] {$\overline {j^\ast}|_{\ker \overline {f_0^\ast}}$}  (m-2-1)
(m-2-1) edge node[auto] {$ $}(m-2-2)
(m-1-2) edge node[auto] {$ \overline {j^\ast}$}(m-2-2)
(m-1-2) edge node[auto] {$\overline {f_0^\ast}$} (m-1-3)
(m-2-2) edge node[auto] {$\overline {f_1^\ast}$} (m-2-3)
(m-1-3) edge node[auto] {$\overline {j'^\ast}$} (m-2-3)
;
\end{tikzpicture}
\end{center}
and this square commutes, hence the map on left-hand side is defined.
We write $K^\ast(j,j'):= \overline {j^\ast}|_{\ker \overline {f_0^\ast}}$ for the map on the left-hand side. In this way,~$K^\ast$ defines a functor $\GACh\longrightarrow \RCh^{\|\cdot\|}$.
\item We write $H^\ast(C^\ast,D^\ast,f^\ast)$ to denote the cohomology of $K^\ast(C^\ast,D^\ast,f^\ast)$ endowed with the induced semi-norm. In this way, we have defined a functor~$\GACh\longrightarrow \Modgrn$. 
\end{enumerate}
\end{def.}

 The main example of $(\G,\A)$-cochain complexes in this section is given by the pair of canonical resolutions for $\G$ and $\A$:

\begin{exa}
 Let $i\colon \A\longrightarrow \G$ be a groupoid pair and $V$ a Banach $\G$-module. Then $C^{\ast}(\G,\A;V):= (B(C_{\ast}(\G),V),B(C_{\ast}(\A),i^{\ast} V), B(C_{\ast}(i),V))$ is a Banach~$(\G,\A)$-cochain complex. By definition 
\begin{equation*}
  H^{\ast}(C^\ast(\G,\A;V)) = H_b^{\ast}(\G,\A;V).
\end{equation*}
\end{exa}
\begin{rem}
 Let $(f_0^\ast, f_1^\ast),(g_0^\ast,g_1^\ast)\colon (C_0^\ast, C_1^\ast, \varphi^\ast) \longrightarrow (D_0^\ast, D_1^\ast, \psi^\ast)$ be a pair of $(\G,\A)$-cochain maps. Let $(h_0^\ast, h_1^\ast)\colon (C_0^\ast, C_1^\ast, \varphi_0^\ast) \longrightarrow (D_0^{\ast-1}, D_1^{\ast-1}, \psi^{\ast-1})$ be a $(\G,\A)$-cochain homotopy between $(f_0^\ast, f_1^\ast)$ and $(g_0^\ast,g_1^\ast)$, i.e., $h_0^\ast$ is a $\G$-cochain homotopy between $f_0^\ast$ and $g_0^\ast$ and $h_1^\ast$ is an $\A$-cochain homotopy between $f_1^\ast$ and $g_1^\ast$, and the pair $(h^\ast_0,h_1^\ast)$ is a family of $(\G,\A)$-maps. Then the pair $(h^\ast_0,h^\ast_1)$ induces an~$\R$-cochain homotopy between $K^\ast(f_0^\ast,f_1^\ast)$ and~$K^\ast(g_0^\ast,g^\ast_1)$. In this sense,  cohomology of $(\G,\A)$-cochain complexes is a homotopy invariant.  
\end{rem}

\begin{def.}[Relatively injective pairs]Let $i\colon \A\longrightarrow \G$ be a pair of groupoids. 
 \begin{enumerate}
  \item A $(\G,\A)$-map $(j,j')\colon (V,V',\varphi)\longrightarrow (W,W',\psi)$ is called \emph{relatively injective}, if there is a (not necessarily $(\G,\A)$-equivariant) \emph{split}
 \[(\sigma,\sigma')\colon (W,W',\psi)\longrightarrow (V,V',\varphi),\]
 such that $(\sigma,\sigma')\circ (j,j')= (\id_V,\id_{V'})$ and $\|\sigma\|_{\infty}\leq 1$ and $\|\sigma'\|_{\infty}\leq 1$. 
\item A $(\G,\A)$-module $(I,I',f)$ is called \emph{relatively injective} if for each relatively injective $(\G,\A)$-map $(j,j')\colon (V,V',\varphi)\longrightarrow (W,W',\psi)$  between $(\G,\A)$-mo\-dules and each $(\G,\A)$-map $(\alpha,\alpha')\colon (V,V',\varphi)\longrightarrow (I,I',f)$, there exists a~$(\G,\A)$-map $(\beta,\beta')\colon (W,W',\psi)\longrightarrow (I,I',f)$, such that~$(\beta,\beta')\circ (j,j') = (\alpha,\alpha')$ and $\|\beta\|_{\infty}\leq \|\alpha\|_{\infty}$ and $\|\beta'\|_{\infty}\leq \|\alpha'\|_{\infty}$. See Figure~\ref{f:relinj}.
 \end{enumerate}
\end{def.}
\begin{figure}
\begin{center}
\begin{tikzpicture}
\matrix (m) [scale = 0.5,matrix of math nodes, row sep=2.5em,
column sep=3em, text height=2.5ex, text depth=0.25ex]
{
i^\ast V & & V'&\\
&i^\ast I&&I'\\
i^\ast W& & W'\\
};
\path[->]
(m-3-1)  edge[bend left=45pt] node[left=2pt, above=20pt] {$ i^\ast\sigma$}(m-1-1)
(m-1-1)  edge node[below=10pt, left] {$ i^\ast j$}(m-3-1)
(m-1-1) edge node[auto] {$i^\ast \alpha $}(m-2-2)
;
\path[color=gray!80!black,->]
(m-1-1)  edge node[auto] {$\varphi $}(m-1-3)
(m-3-1) edge node[auto] {$\psi $}(m-3-3)
(m-2-2) edge node[left = 14pt, above] {$ f$}(m-2-4)
(m-1-3) edge[-,line width=4pt,draw=white] node[below=10pt, right] {$ $}(m-3-3)
(m-1-3) edge node[below=10pt, right] {$ j'$}(m-3-3)
(m-3-3)  edge[bend left=45pt,line width=4pt,draw=white] node[left=1pt, above=20pt] {$ $}(m-1-3)
(m-3-3)  edge[bend left=45pt] node[left=1pt, above=20pt] {$\sigma'$}(m-1-3)
(m-1-3) edge node[auto] {$ \alpha'$}(m-2-4)
;
\path[dashed,->]
(m-3-1) edge node[auto] {$ i^\ast\beta$}(m-2-2);
\path[color=gray!80!black,dashed,->]
(m-3-3) edge node[below=3pt,right] {$ \beta'$}(m-2-4)
;

\end{tikzpicture}\caption{Extension Problem for Pairs.}\label{f:relinj}
\end{center}
\end{figure}
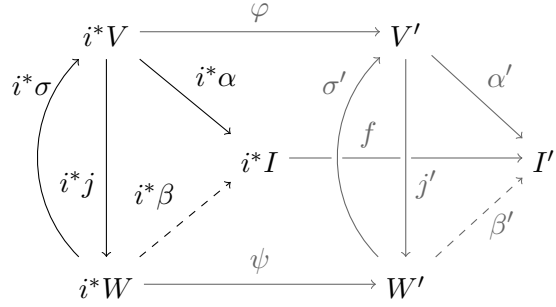

\begin{rem}
Park~\cite{Pa03} treats bounded cohomology of a pair~$(G,A)$ of groups via so called \emph{allowable resolutions}, i.e., via pairs consisting of a strong, relatively injective $G$-resolution and a strong, relatively injective $A$-resolution, together with a cochain map that commutes with a pair of norm non-increasing cochain contractions of the two resolutions. Thus, the condition on relative injectivity of an allowable pair is a priori weaker than our definition of relatively injective~$(G,A)$-resolutions. However, as noted by Frigerio and Pagliantini~\cite[Remark 3.8]{FP}, it is not clear why Park's definition should lead to a version of the fundamental lemma. Our definition avoids this problem and still includes the interesting examples. 
\end{rem}

\begin{exa}
Let $i\colon \A\longrightarrow \G$ be a pair of groupoids and $U$ a Banach $\G$-module. For each $n\in \N$, the Banach $(\G,\A)$-module $C^n(\G,\A;U)$ is relatively injective. 
\end{exa}
\begin{proof}
By a straightforward calculation~\cite[Example 3.5.21]{Bl14} one shows that the concrete maps constructed in the proof of Proposition~\ref{p:boundedrelinj} can be used to solve also the relative extension problem. 
\end{proof}

\begin{def.}
 Let $i\colon \A\longrightarrow \G$ be a pair of groupoids, $(C^\ast, D^\ast,f^\ast)$ be a Banach $(\G,\A)$-cochain complex and $V$ be a Banach $\G$-module. Furthermore,  let~$(\varepsilon,\varepsilon')\colon (V,i^\ast V,\id_{i^\ast V})\longrightarrow (C^0,D^0,f^0)$ be a $(\G,\A)$-augmentation map. We call $(C^\ast,D^\ast,f^\ast, (\varepsilon,\varepsilon'))$ a \emph{strong resolution of} $V$ if there exists a (not necessarily $(\G,\A)$-equivariant) norm non-increasing cochain contraction $(s^\ast,t^\ast)$ of $(C^\ast,D^\ast,f^\ast, (\varepsilon,\varepsilon'))$, i.e., a norm non-increasing contraction $s^\ast$ of $(C^\ast,\varepsilon)$ and a norm non-increasing contraction $t^\ast$ of~$(D^{\ast},\varepsilon')$, such that $f^\ast$ commutes with $i^\ast s^\ast$ and $t^\ast$. 
\end{def.}

\begin{exa}
 Let $i\colon \A\longrightarrow \G$ be a pair of groupoids and $V$ be a Banach $\G$-module. Let $s_\G^\ast$ and $s_\A^\ast$ be the norm-non-increasing cochain contractions for $(B(C_{\ast}(\G),V),\varepsilon_{\G})$ and $(B(C_{\ast}(\A),i^{\ast} V),\varepsilon_{\A})$ as in Example~\ref{e:cochaincontr}. By a short calculation, we see that $(s_{\G}^\ast,s_{\A}^\ast)$ is a norm non-increasing cochain contraction of~$(C^{\ast}(\G,\A;V),(\varepsilon_{\G},\varepsilon_{\A}))$. Hence, $(C^{\ast}(\G,\A;V),(\varepsilon_{\G},\varepsilon_{\A}))$ is a strong $(\G,\A)$-resolution of $V$, called \emph{the standard $(\G,\A)$-resolution of $V$}.
\end{exa}

\begin{prop}[Fundamental Lemma for Pairs]\label{p:fundamentalpairs} Let $i\colon \A\longrightarrow \G$ be a pair of groupoids.
 Let $(I^\ast,J^\ast,\varphi^\ast)$ be a relatively injective $(\G,\A)$-cochain complex and $(\varepsilon,\varepsilon')\colon (W,i^\ast W,\id_{i^\ast W})\longrightarrow (I^0,J^0,\varphi^0)$  a $(\G,\A)$-augmentation map. Let~$(C^\ast, D^\ast,\psi^\ast,(\nu,\nu'))$ be a strong $(\G,\A)$-resolution of a Banach $\G$-module $V$. Let $(f,f')\colon (V,i^{\ast}V,\id_{i^\ast V})\longrightarrow (W,i^\ast W,\id_{i^\ast W})$ be a $(\G,\A)$-map. Then there exists an extension of~$(f,f')$ to a bounded $(\G,\A)$-cochain map from the resolution~$(C^\ast, D^\ast,\psi^\ast,(\nu,\nu'))$ to the augmented cochain complex~$(I^\ast,J^\ast,\varphi^\ast, (\varepsilon,\varepsilon'))$. This extension is unique up to bounded $(\G,\A)$-cochain homotopy.
\end{prop}
\begin{proof}
 The proof is the same as the one of Proposition~\ref{t:fundamental}, just solving the extension problems for pairs instead of for single modules.
\end{proof}

\begin{prop}\label{p:normnonincrpair}
 Let $i\colon \A\longrightarrow \G$ be a groupoid pair and $V$ a Banach $\G$-module. Let $(C^\ast,D^\ast,\varphi^\ast,(\nu,\nu'))$ be a strong $(\G,\A)$-resolution of $V$. Then for each norm non-increasing cochain contraction of $(C^\ast,D^\ast,\varphi^\ast,(\nu,\nu'))$  there exists a canonical, norm non-increasing $(\G,\A)$-cochain map 
\[
 (C^\ast,D^\ast,\varphi^\ast)\longrightarrow C^\ast(\G,\A;V),
\]
extending $(\id_{V},\id_{i^\ast V})$. 
\end{prop}
\begin{proof}
 The proof is basically the same as the one for Proposition~\ref{p:normnonincr}, one just has to check that the constructed pair of maps is a $(\G,\A)$-map.
\end{proof}

\begin{cor}\label{c:relviastronginj}
 Let $i\colon \A\longrightarrow \G$ be a groupoid pair and $V$ a Banach $\G$-module. Let $(C^\ast,D^\ast,\varphi^\ast,(\nu,\nu'))$ be a strong, relatively injective $(\G,\A)$-resolu\-tion~of $V$. Then there exists a canonical, semi-norm non-increasing isomorphism of graded $\R$-modules
\[
 H^\ast(C^\ast,D^\ast,\varphi^\ast) \longrightarrow H^\ast_b(\G,\A;V).
\]
\end{cor}
\begin{proof}
 By Proposition~\ref{p:normnonincrpair},  there is a norm non-increasing $(\G,\A)$-cochain map~$(C^\ast,D^\ast,\varphi^\ast)\longrightarrow C^\ast(\G,\A;V)$, extending $(\id_{V},\id_{i^\ast V})$. Hence this map induces  a semi-norm non-increa\-sing map  $H^\ast(C^\ast,D^\ast,\varphi^\ast) \longrightarrow H^\ast_b(\G,\A;V).$ By the fundamental lemma for pairs, this map is an isomorphism in each degree and does not depend of the choice of the lift of $(\id_{V},\id_{i^\ast V})$. 
\end{proof}

\section{Amenability}
In this section, we define amenability for groupoids, similar to the definition of amenability of measured groupoids~\cite{AR99}.
\begin{def.}
 Let $\G$ be a groupoid and $V$ a normed $\G$-module. We define a normed $\G$-module $\ell^{\infty}(\G,V)$ by:
\begin{enumerate}
 \item For all $e\in\ob\G$ we set $\ell^{\infty}(\G,V)_e:= \ell^{\infty}(\G_{e},V_e),$ endowed with the $\|\cdot\|_{\infty}$-norm. Here, we denote by $\G_{e}$ the set of all morphisms in $\G$ ending in $e$, i.e., $\G_e = t^{-1}(e)$.
\item We define the $\G$-action by setting for all $g\in\G$
\begin{align*}
 \rho_{g}\colon \ell^{\infty}(\G,V)_{s(g)}&\longrightarrow \ell^{\infty}(G,V)_{t(g)}\\
\varphi&\longmapsto \bigl(h\longmapsto g\cdot \varphi(g^{-1}\cdot h)).
\end{align*}
\end{enumerate}
\end{def.}
\begin{rem}
 If $\G$ is a groupoid and $V$ a normed $\G$-module, there is a canonical $\G$-inclusion map $c_V \colon V\longrightarrow \ell^{\infty}(\G,V)$, corresponding to viewing elements in~$V$ as constant functions $\G\longrightarrow V$, given by setting for all~$e\in\ob\G$
\begin{align*}
 c_{Ve}\colon V_e&\longrightarrow \ell^{\infty}(\G,V)_e\\
 a &\longmapsto (g\longmapsto a).
\end{align*}
\end{rem}

\begin{def.}\label{d:amenablegroupoid}
Let $\G$ be a groupoid and $V$ a normed $\G$-module. 
\begin{enumerate}
 \item An \emph{equivariant mean on $\G$ with coefficients in $V$} is a $\G$-morphism 
\[
 m\colon \ell^{\infty}(\G,V)\longrightarrow V,
\]
satisfying $\|m\|_{\infty}=1$ and $m\circ c_V =\id_{V}$.
\item  If $V=\R_\G$, we call $m$ simply a \emph{(left)-invariant mean on $\G$} and we call $\G$ \emph{amenable} if there exists a (left-)invariant mean on $\G$. 
\end{enumerate}
\end{def.}
This definition clearly extends the definition of amenability for groups~\cite{PA88}. 
\begin{exa}
 Let $(A_i)_{i\in I}$ be a family of groups. Then $\amalg_{i\in I} A_i$ is amenable if and only if all $A_i$ are amenable groups.
\end{exa}
\begin{exa}
 Let $G$ be an amenable group acting on a set $X$. Then the action groupoid $G\ltimes X$ is amenable (Example~\ref{e:actiongrp}). 
\end{exa}
Starting from an invariant mean on $\G$, we can actually construct invariant means with coefficients in any dual $\G$-module:
\begin{prop}
 Let $\G$ be an amenable groupoid and $V$ a normed $\G$-module. Then there exists an equivariant mean on $\G$ with coefficients in ~$V'$
\end{prop}
\begin{proof}
Let $m_\R$ be a left-invariant mean on $\G$. Then the map 
 \begin{align*}
 m_V \colon \ell^\infty(\G,V')&\longrightarrow V'\\
 \varphi &\longmapsto \Bigl(v\longmapsto m_\R \bigl(g\longmapsto \varphi(g)(v)\bigr)\Bigr).
\end{align*}
is clearly an equivariant mean on $\G$ with coefficients in $V'$. 
\end{proof}
\section{The Algebraic Mapping Theorem}\label{s:algmap}
In this section, we show  that the bounded cohomology of a groupoid $\G$ relative to an amenable groupoid is equal to the bounded cohomology of $\G$ and discuss a characterisation of amenability in terms of bounded cohomology. The techniques used are similar to the group case~\cite{Mo01,No90}.

Let $i\colon \A\longrightarrow \G$ be a pair of groupoids. Recall that we write $L_\ast(\G)$ to denote the homogeneous Bar resolution of $\G$,  Definition~\ref{d:hombar}. Let $V$ be a Banach~$\G$-module. Denote by $K_L^\ast(\G,\A;V)$ the kernel of the map \[B(L_\ast(i),V) \colon B(L_\ast(\G),V)\longmapsto B(L_\ast(\A),i^\ast V).\] 
and write $j^\ast\colon K_L^\ast(\G,\A;V)\longhookrightarrow B(L_\ast(\G),V)$ to denote the canonical inclusion. Since by Proposition~\ref{p:homogeneous} there is a natural and canonical isometric isomorphism between $L_{\ast}(\G)$ and $C_{\ast}(\G)$, there is a canonical isometric isomorphism between $K_L^\ast(\G,\A;V)$ and $K^\ast(\G,\A;V)$. Thus we can also use the former to calculate bounded cohomology of the pair $(\G,\A)$. 
\begin{prop} Let $\G$ be a groupoid. 
 The following map is a norm non-increasing $\G$-chain map extending $\id_{\R_\G}$
\begin{align*}
 \Alt_n \colon L_n(\G) &\longrightarrow L_n(\G)\\
 (g_0,\dots,g_n) &\longmapsto \frac{1}{(n+1)!}\cdot \sum_{\sigma\in \Sigma_{n+1}} \sgn(\sigma)\cdot (g_{\sigma_0},\dots,g_{\sigma_n})
\end{align*}
For a Banach $\G$-module $V$, we write $\Alt^n_V\colon B(L_n(\G),V) \longrightarrow B(L_n(\G),V)$ for the corresponding dual map. 
\end{prop}
\begin{proof}
This is well-known for groups~\cite[Section 7.4]{Mo01} and easily verified. 
\end{proof}
\begin{prop}
 Let $\G$ be a groupoid and $\A\subset \G$ be an amenable subgroupoid.  Let $V$ be a Banach $\G$-module and $m_{\A}$ be an equivariant mean on~$\A$ with coefficients in $V'$ (with the induced $\A$-module structure). For each~$n\in \N$ and $i\in\{0,\dots,n\}$, define a map $\Phi_i^n \colon B(L_n(\G),V')\longrightarrow B(L_n(\G),V')$ by setting for all $\varphi \in B(L_n(\G),V')$ and $(g_0,\dots, g_0)\in S_n(\G)$
\[ 
\Phi^{n}_i(\varphi) (g_0,\dots,g_n) = g_i\cdot m_{A_{s(g_i)}} \bigl(a_i\longmapsto  g_i^{-1}\cdot \varphi(g_0,\dots,g_{i-1},g_i\cdot a_i,g_{i+1},\dots,g_n)\bigr)
\]
if $s(g_i)\in \ob\A$, and $\Phi^{n}_i(\varphi)(g_0,\dots,g_n) = \varphi(g_0,\dots,g_n)$ else. 

For each $n\in\N$, set $A_{\A}^n:= \Phi^n_0\circ\cdots\circ\Phi^n_n.$ Then $A_{\A}^\ast$ is a norm non-increasing~$\G$-cochain map extending $\id_{V'}$.  
\end{prop}
\begin{proof}\hfill
 \begin{enumerate}
  \item The map $A_\A^\ast$  extends $\id_{V'}$ because the mean of a constant function is equal to the value of the function. 
\item  It follows directly that for all $n\in\N_{>0}$
\[
 \fa{i,j\in \{0,\dots,n\}} \Phi^n_j\circ \delta_i^{n-1} = \begin{cases}
                                                        \delta_i^{n-1} \circ \Phi^{n-1}_j &\text{if } i>j\\
							\delta^{n-1}_i \circ \Phi_{j-1}^{n-1} &\text{if } i<j\\
							\delta_i^{n-1} &\text{if } i=j.
                                                       \end{cases}
\]
Thus $A^\ast_\A$ is a cochain map.
\item  The map $A_\A^\ast$ is norm non-increasing because means are norm non-increa\-sing. 
\item Clearly all $\Phi_i^n$ are  $\G$-equivariant, hence $A_{\A}^\ast$ is also $\G$-equivariant.\qedhere

 \end{enumerate}
\end{proof}
\begin{prop}\label{p:meanfactorizes}
Let $\G$ be a groupoid and $\A\subset\G$ be an amenable subgroupoid. Let $V$ be a Banach $\G$-module. Then the composition \[\Alt_{V'}^\ast \circ A_{\A}^\ast \colon B(L_\ast(\G),V') \longrightarrow B(L_{\ast}(\G),V')\] is a norm non-increasing $\G$-map extending $\id_{V'}$. For $n\in\N_{\geq 1}$, it factors through~$K_{L}^n(\G,\A;V')$, i.e., the following diagram commutes: 
    \begin{center}
\begin{tikzpicture}
\matrix (m) [scale = 0.5,matrix of math nodes, row sep=3em,
column sep=1em, text height=1.5ex, text depth=0.25ex]
{
B(L_n(\G), V') & &B(L_n(\G), V')\\
&K_{L}^n(\G,\A;V')
\\ };
\path[->]
(m-1-1) edge node[auto] {$\Alt^n_{V'} \circ A_\A^n$}(m-1-3)
(m-1-1) edge node[auto] {$  $}(m-2-2)
(m-2-2) edge node[right=5pt] {$ j^n$}(m-1-3)
;
\end{tikzpicture}
\end{center}
\end{prop}
\begin{proof}
We have to show that $B(L_n(i),V')\circ \Alt^n_{V'}  \circ A_\A^n=0$ for all $n\in\N_{\geq 1}$. But for all $e\in\ob\A$ and $(a'_0,\dots,a'_n)\in \A_e^{n+1}$, $b_i'\in A_e$ and $\varphi \in B(L_n(\G),V')_e$
\begin{align*}
\Phi_i^n&(\varphi)(a_0',\dots,a_n')\\ &= a_i'\cdot m_{A_s(a'_i)} \bigl(a_i\longmapsto {a'}_i^{-1} \cdot \varphi(a'_0,\dots,a'_i\cdot a_i,\dots,a'_n)\bigr)\\
 &= a_i'\cdot ({b_i'}^{-1}\cdot {a_i'})^{-1} \\
&\phantom{=}\cdot m_{A_s(b'_i)}\bigl(a_i\longmapsto ({b_i'}^{-1}\cdot {a_i'})\cdot {a'}_i^{-1} \cdot \varphi(a'_0,\dots,a'_i\cdot ({b_i'}^{-1}\cdot {a_i'})^{-1} a_i,\dots,a'_n)\bigr)\\
&=  {b_i'}\cdot m_{A_{s(b_i')}} \bigl(a_i\longmapsto  {b_i'}^{-1}\cdot \varphi(a'_0,\dots, b_i'\cdot a_i,\dots,a'_n)\bigr)\\
&= \Phi_i^n(\varphi)(a_0',\dots,b_i',\dots,a_n')
\end{align*}
Hence $\Phi_i^n(\varphi)(a'_0,\dots,a'_n)$ does not depend on $a_i'$. Therefore, $A_{\A}^n(\varphi\circ L_n(i))$ is constant for all $\varphi\in B(L_n(\G),V')$ and in particular invariant under permutations of the arguments. Thus, $B(L_n(i),V')\circ \Alt^n_{V'}  \circ A_\A^n=0$
\end{proof}
\begin{cor}\label{c:amenablevanish}
 Let $\G$ be an amenable groupoid and $V$ a Banach $\G$-module. Then $H_b^n(\G,V') =0$ for all $n\in\N_{\geq 1}$.
\end{cor}
\begin{proof}
 Since the composition $\Alt_{V'}^\ast \circ A_{\G}^\ast \colon B(L_n(\G),V') \longrightarrow B(L_{\ast}(\G),V')$ is a $\G$-map extending $\id_{V'}$, by the fundamental lemma, Theorem~\ref{t:fundamental}, it induces the identity on bounded cohomology. On the other hand, by Proposition~\ref{p:meanfactorizes}, it factors through the trivial complex $K_L^\ast(\G,\G;V')$ in degree greater or equal~1 and hence $H_b^{\ast}(\G; V') =0$ for all $n\in\N_{\geq 1}$. 
\end{proof}

\begin{cor}[Algebraic Mapping Theorem]\label{c:algebraicmappingthm}
 Let $i\colon \A\longhookrightarrow\G$ be a pair of groupoids such that $\A$ is amenable. Let $V$ be Banach $\G$-module. Then 
\[
 H^n(j^\ast)\colon H_b^n(\G,\A;V') \longrightarrow H_b^n(\G;V')
\]
is an isometric isomorphism for each $n\in\N_{\geq 2}$.
\end{cor}
\begin{proof}
The map $j^\ast$ is norm non-increasing and by Corollary~\ref{c:amenablevanish} and the long exact sequence, induces a norm non-increasing isomorphism in bounded cohomology in degree $n\in \N_{\geq 2}$. Since $\Alt_{V'}^\ast \circ A_{\A}^\ast$ is a $\G$-map extending~$\id_{V'}$, by the fundamental lemma, Theorem~\ref{t:fundamental}, it induces the identity on bounded cohomology. The map $ B(L_n(\G),V')\longrightarrow K_L^n(\G,\A,V')$ induced by~$\Alt_{V'}^\ast \circ A_{\A}^\ast$ in degree $n\in\N_{\geq 1}$ is also norm non-increasing. Therefore, by Proposition~\ref{p:meanfactorizes},  $H^n(j^\ast)$ is an isometric isomorphism for $n\in\N_{\geq 2}$.  
\end{proof}

\begin{rem}
Even when considering only groups, the algebraic mapping theorem is wrong in general  if we use coefficients in general modules, not just dual spaces~\cite[Section 7.2 and 7.5]{No90}. 
\end{rem}
As for groups, the converse to Corollary~\ref{c:amenablevanish}  holds in the strong sense that vanishing in degree 1 for a certain dual space is sufficient for a groupoid to be amenable. Indeed, consider the Banach space $\Sigma \R_\G:=\coker c_{\R_\G}$:
\begin{prop}\label{p:amenablebounded}
 Let $\G$ be a groupoid. If $H_b^1(\G;(\Sigma \R_\G)') =0 $, then $\G$ is amenable. 
\end{prop}
\begin{proof}
 The proof~\cite[Proposition 4.2.7]{Bl14} is basically the same as Noskov's proof~\cite[Section 7.1]{No90} in the group case.
\end{proof}

\section{The Topological Resolution}
In this section, let $X$ be a CW-complex. Let $p\colon \widetilde X\longrightarrow X$ be the universal cover. For each $n\in \N$, write $S_n(\widetilde X)$ for the set of singular $n$-simplices in~$\widetilde X$. The group~$\Deck(X)$ of deck transformations of $p$ acts (from the left) on $S_n(\widetilde X)$, $C^{\text{sing}}_n(\widetilde X)$ and $\widetilde X$ as usual. Again, we will denote the fundamental groupoid of~$X$  by $\pi_1(X)$. For each~$x\in\widetilde X$, denote the connected component of $\widetilde X$ containing~$x$ by $\widetilde X_x$. For $\sigma \in S_n(X)$, we will also write $\sigma_0,\dots,\sigma_n\in X$ to denote the vertices of $\sigma$.

\begin{def.}
 Let $X$ be a CW-complex. There is a partial action of~$\pi_1(X)$ on $\widetilde X$, i.e., for each $\gamma \in \pi_1(X)$ there is a bijection
\begin{align*}
 \rho_{\gamma} \colon p^{-1}(s(\gamma)) &\longrightarrow p^{-1}(t(\gamma))\\
 x&\longmapsto \gamma\cdot x:= \widetilde{\gamma_{x}}(1).
\end{align*}
Here, for each $x\in\widetilde X$, let $\widetilde{\gamma_x}$ denote the lift of a representative of $\gamma$ starting at~$x$. Moreover, we have $ \rho_{\id_x} = \id_{p^{-1}(x)}$ for all $x\in X$ and~$ \rho_{\gamma'\cdot \gamma} = \rho_{\gamma'}\circ \rho_{\gamma}$ for all $\gamma,\gamma'\in \pi_1(X)$ with $s(\gamma')=t(\gamma)$.
This induces a $\pi_1(X)$-module structure on $\R[\widetilde{X}]:= \bigoplus_{\widetilde X}\R$.

\end{def.}

\begin{lemma}\label{l:tgaction}
 For all points $x\in \widetilde X$, all $\gamma\in \pi_1(X)$ with $s(\gamma)=p(x)$ and all~$\beta\in \Deck(X)$ we have
\[
 \beta(\gamma\cdot x) = \gamma\cdot\beta(x).
\]
In other words, $\R[\widetilde X]$ is a $(\Deck(X),\pi_1(X))$-bi-module, i.e., represented by a functor
$\pi_1(X) \longrightarrow \DeckMod$. 
\end{lemma}
\begin{proof}
If $\widetilde{\gamma_x}$ is a lift of a representative of $\gamma$ starting in $x$, then $\beta\cdot\widetilde{\gamma_x}$ is a  lift of a representative of $\gamma$ starting in $\beta(x)$ hence
\[
 \gamma \cdot (\beta(x)) = \widetilde {\gamma_{\beta(x)}}(1)\\
 = \beta (\widetilde{\gamma_x}(1))\\
= \beta (\gamma\cdot x). \qedhere
\]
\end{proof}
\noindent Now, we introduce a $\pi_1(X)$-version of the $\pi_1(X,x)$-chain complex $C_{\ast}^{\sing}(\widetilde X)$:

\begin{def.}\hfill
\begin{enumerate}
 \item  For all $n\in\N$, we define the set 
\begin{align*}
 Q_n(X):= \frac{\{(\sigma, x)\in S_n(\widetilde X)\times \widetilde X\mid \im \sigma \subset \widetilde X_x\} }{\Bigl((\beta\cdot \sigma, \beta\cdot x) \sim (\sigma, x) \mid \beta\in\Deck(X)\Bigl)}.
\end{align*}
By a slight abuse of notation, we denote the class of a $(\sigma,x)\in S_n(\widetilde X)\times \widetilde X$ in $Q_n(X)$ also by $(\sigma,x)$. For all $e\in X$ we set \[Q_n(X)_e:= \{(\sigma,x)\in Q_n(X)\mid p(x)=e\}.\] 
\item For all $n\in\N$ we define a normed $\pi_1(X)$-module $C_n(X)$ via
\[
 C_n(X) = (\R\langle Q_n(X)_e\rangle)_{e\in X}
\]
with the partial $\pi_1(X)$-action given by setting for each $\gamma\in \pi_1(X)$
\begin{align*}
 \rho_\gamma \colon C_n(X)_{s(\gamma)} &\longrightarrow C_n(X)_{t(\gamma)}\\
 (\sigma, x)&\longmapsto (\sigma, \gamma\cdot x).
\end{align*}
This action is well-defined by Lemma~\ref{l:tgaction}. Finally, we endow $C_\ast(X)$ with the $\ell^1$-norm with respect to the basis $Q_\ast(X)$.
\item For all $n\in\N_{>0}$, we define boundary maps $\partial_n\colon C_n(X)\longrightarrow C_{n-1}(X)$ via 
\[
 \partial_n(\sigma,x) := \sum_{i=0}^n (-1)^i (\partial_{i,n}\sigma,x).
\]
Here, we write $\partial_{i,n}$ to denote the usual $i$-th face map in dimension $n$. These maps are clearly $\pi_1(X)$-equivariant and well-defined since the action by deck transformations is compatible with the usual boundary maps. The $\partial_n$ are bounded and for every $n\in\N$, we have~$\|\partial_n\|_{\infty} \leq n+1$. Furthermore, these are obviously boundary maps.
\item Additionally, we will consider the canonical augmentation map
\begin{align*}
 \varepsilon \colon C_0(X)&\longrightarrow \R[X]\\
 (\sigma,x)&\longmapsto p(x)\cdot 1. 
\end{align*}
Here, we denote by $\R[X]$ the trivial $\pi_1(X)$-module $\R_{\pi_1(X)}$.

\item   For any subset $\emptyset\neq I\subset X$, we analogously define a normed~$\pi_1(X,I)$-chain complex $(C_\ast(X:I))_{\ast\in\N}$.
\end{enumerate}
\end{def.}

\begin{rem}
 We can make this definition more concise using tensor products over group modules and the $(\Deck(X),\pi_1(X))$-bi-module structure on~$\R[\widetilde X]$. Then $C_n(X)$ is nothing else than $C_n^{\sing}(\widetilde X)\otimes_{\Deck(X)} \R[\widetilde X]$, i.e., given by the composition
\begin{center}
\begin{tikzpicture}
\matrix (m) [scale = 0.5,matrix of math nodes, row sep=2.5em,
column sep=4em, text height=2.5ex, text depth=0.25ex]
{\pi_1(X)& \DeckMod &&\Mod.
\\ };
\path[->]
(m-1-1) edge node[auto] {$\R[\widetilde X]$}(m-1-2)
(m-1-2) edge node[auto] {$\otimes_{\Deck(X)} C_n^{\sing}(\widetilde X)$}  (m-1-4)
;
\end{tikzpicture}
\end{center}
\end{rem}

\begin{rem}
 Let $X$ be a connected topological space and $x\in X$ a point. Then $C_{\ast}(X:\{x\}) = C_{\ast}^{\sing}(\widetilde X;\R)$ as $\pi_1(X,x)$-modules. Thus the definition, and the ones that will follow, generalise the classical situation. 
\end{rem}
\begin{prop}[Functoriality]
 Let $f\colon X\longrightarrow Y$ be a continuous map between two CW-complexes. Let $\widetilde f \colon \widetilde X\longrightarrow \widetilde Y$ be a lift. Then, $\widetilde f$ is $\pi_1(X)$-equivariant, i.e., for all~$\gamma \in \pi_1(X)$ and all $x\in p_X^{-1}(s(\gamma))$ we have
\[
 \widetilde f(\gamma \cdot x) = \pi_1(f)(\gamma)\cdot \widetilde f(x). 
\]

Therefore, the following map is a $\pi_1(X)$-chain map with respect to the induced $\pi_1(X)$-structure on $C_\ast(Y)$.
\begin{align*}
 C_\ast(f)\colon C_\ast(X)&\longrightarrow C_\ast(Y)\\
 (\sigma,x) &\longmapsto (\widetilde f\circ \sigma, \widetilde f(x)).
\end{align*}
The map $C_\ast(f)$ does not depend on the choice of the lift $\widetilde f$. In particular,~$C_\ast$  is functorial in the sense that $C_\ast(g\circ f) = (\pi_1(f)^{\ast}C_\ast(g))\circ C_\ast(f)$ for all continuous maps $f\colon X\longrightarrow Y$ and $g\colon Y\longrightarrow Z$. 
\end{prop}
\begin{proof}
Let $\widetilde \gamma_x$ be a lift of $\gamma$ starting at $x$. Then 
\[
 p_Y\circ \widetilde f\circ \widetilde \gamma_x = f\circ p_X \circ \widetilde \gamma_x = f\circ \gamma = \pi_1(f)(\gamma). 
\]
Hence $\widetilde f\circ \widetilde \gamma_x$ is a lift of $\pi_1(f)(\gamma)$ starting at $\widetilde f(x)$ and therefore
\[
\widetilde f(\gamma\cdot x) = \widetilde f(\widetilde \gamma_x(1)) = (\widetilde f\circ \widetilde \gamma_x)(1) = \pi_1(f)(\gamma)\cdot \widetilde f(x).  
\]
The lifts of $f$ are exactly given by~$\{\beta\circ \widetilde f\mid \beta \in \Deck(Y)\}$.

For all~$\beta\in \Deck(X)$ the map $\widetilde f\circ \beta$ is a lift of $f$, and therefore there exists a $\beta'\in \Deck(Y)$ such that~$\beta'\circ \widetilde f= \widetilde f\circ \beta$, hence for all $(\sigma,x)\in Q_n(X)$

\[
 (\widetilde f\circ(\beta\cdot\sigma), \widetilde f(\beta\cdot x)) =  (\beta'\cdot(\widetilde f\circ \sigma), \beta'\cdot \widetilde f(x))= (\widetilde f\circ \sigma,\widetilde f(x)).
\]
Thus $C_n(f)$ is well-defined. Clearly, for all $\beta\in \Deck(Y)$, the lifts $\widetilde f$ and $\beta\cdot \widetilde f$ induce the same map $C_n(f)$, hence $C_n(f)$ does not depend on the choice of the lift. Therefore, the construction is functorial since the composition of the lifts of two maps is a lift of the composition of the two maps. 
The family $C_{\ast}(f)$ is a chain map since $C^{\sing}_\ast(\widetilde f)$ is a chain map. Furthermore, by definition it is compatible with the induced $\pi_1(X)$-structure on $C_\ast(Y)$. 
\end{proof}

\section{Bounded Cohomology of Topological Spaces}
In this section, we use the $\pi_1(X)$-chain complex $C_{\ast}(X)$  to define bounded cohomology of $X$ with twisted coefficients in a Banach $\pi_1(X)$-module $V$, slightly extending the usual definition of bounded cohomology with twisted coefficients. We will then study the normed cochain complex~$B(C_\ast(X),V')$, show that it is a strong, relatively injec\-tive~$\pi_1(X)$-resolution of $V'$ and derive the absolute mapping theorem for groupoids, i.e., that the bounded cohomology of $X$ is isometrically isomorphic to the bounded cohomology of~$\pi_1(X)$, preparing the ground for our proof of the relative mapping theorem in the next section.  
\subsection{Bounded Cohomology of Topological Spaces}

Similar to the case of bounded cohomology of a groupoid, we begin by defining the domain category for bounded cohomology of a space:
\begin{def.}[Domain categories for bounded cohomology of spaces]
We define a category $\TopBanc$ by setting:
\begin{enumerate}
\item Objects in $\TopBanc$ are pairs $(X,V)$, where $X$ is a topological space and~$V$ is a Banach $\pi_1(X)$-module.
\item A morphism $(X,V)\longrightarrow (Y,W)$ in the category $\TopBanc$ is a pair $(f,\varphi)$, where $f\colon X\longrightarrow Y$ is a continuous map and $\varphi \colon \pi_1(f)^{\ast}W\longrightarrow V$ is a bounded $\pi_1(X)$-map. 
\item We define the composition in $\TopBanc$ as follows: For each pair of morphisms  $(f,\varphi) \colon (X,U)\longrightarrow (Y,V)$, $(g,\psi)\colon (Y,V)\longrightarrow (Z,W)$ set 
\[
 (g,\psi)\circ (f,\varphi) := (g\circ f, \varphi\circ (\pi_1(f)^{\ast}\psi)). 
\]

\end{enumerate}
\end{def.}

\begin{def.}[The Banach Bar Complex with coefficients]
 Let $X$ be a CW-complex and $V$ a Banach $\pi_1(X)$-module. Then:
\begin{enumerate}
\item We write 
\[
 C^{\ast}_b(X;V):= B_{\pi_1(X)}(C_{\ast}(X),V).
\]
Together with $\|\cdot\|_{\infty}$, this is a normed $\R$-cochain complex. 
\item If $(f,\varphi)\colon (X,V)\longrightarrow (Y,W)$ is a morphism in $\TopBanc$, we write $ C^{\ast}_b(f,\varphi)$ for the $\R$-cochain map
\begin{align*}
 C^{\ast}_b(Y;W)&\longrightarrow C^{\ast}_b(X;V)\\
 \alpha &\longmapsto \varphi \circ (\pi_1(f)^{\ast}\alpha) \circ C_{\ast}(f).
\end{align*}
\end{enumerate}
This defines a contravariant functor $C^\ast_b \colon \TopBanc \longrightarrow \RCh^{\|\cdot\|}$.
\end{def.}

\begin{def.}
 Let $X$ be a CW-complex, $V$ a Banach $\pi_1(X)$-module.
 We call the cohomology 
\[
 H^{\ast}_{b}(X;V):= H^{\ast}(C_b^{\ast}(X;V)),
\]
together with the induced semi-norm on $H^{\ast}_{b}(X;V)$, the \emph{bounded cohomology of~$X$ with twisted coefficients in $V$}. 

This defines a contravariant functor $H^\ast_b\colon\TopBanc \longrightarrow \Modgrn$.
\end{def.}
Bounded cohomology with twisted coefficients is normally defined only for connected spaces and for general spaces only with trivial coefficients. We show now that our definition coincides in this cases with the usual one. 
\begin{lemma}\label{l:topconn}
Let $X$ be a connected CW-complex, $x\in X$ and $V$ a Banach~$\pi_1(X)$-module. Then there is a canonical isometric $\R$-cochain isomorphism
\begin{align*}
 S^\ast_{X,V}\colon C_b^{\ast}(X;V) &\longrightarrow C_b^{\ast}(X:\{x\};V_x)\\
 (\varphi_e)_{e\in X} &\longmapsto \varphi_x.
\end{align*}
This is natural in $\TopBanc$.
\end{lemma}
\begin{proof}
 Because $\pi_1(X)$ is connected, the map $e\longmapsto \|\varphi_e\|_{\infty}$ is constant for all~$(\varphi_e)_{e\in X}\in C_b^\ast(X;V)$, hence the map $S_{X,V}^\ast\colon (\varphi_e)_{e\in X}\longmapsto \varphi_x$ is isometric and injective. It is a cochain map, since the coboundary operators are defined pointwise. An inverse is given by choosing for each $y\in X$ a path $\gamma_y\in \Mor_{\pi_1(X)}(x,y)$ from $x$ to $y$ and setting 
\begin{align*}
 R^{\ast}\colon C_b^{\ast}(X:\{x\};V_x) &\longrightarrow C_b^{\ast}(X;V)\\
  \varphi&\longmapsto (\gamma_y\cdot \varphi)_{y\in X}.
\end{align*}
The map $R^{\ast}$ does not depend on the choice of $(\gamma_y)_{y\in X}$ since by the $\pi_1(X,x)$-invariance of $\varphi$, we have for all $\gamma \in \Mor_{\pi_1(X)}(x,y)$ 
\[
 \gamma \cdot \varphi = \gamma_y\cdot (\gamma_y^{-1}\cdot \gamma) \cdot \varphi = \gamma_y\cdot \varphi.
\]
In particular, for all $\varphi\in C_b^{\ast}(X:\{x\};V_x)$ the map $R^{\ast}(\varphi)$ is $\pi_1(X)$-invariant and clearly bounded. 
\end{proof}

Similar to Proposition~\ref{p:grpadditive}, we see that bounded cohomology is additive: 
\begin{prop}[Bounded cohomology and disjoint unions of spaces]\label{p:topadditive}
 Let~$X$ be a CW-complex and  $X=\amalg_{\lambda\in \Lambda} X^{\lambda}$ be the partition in connected components. Let $V$ be a Banach $\pi_1(X)$-module. For $\lambda\in\Lambda$, write \text{$V^\lambda$ for} the induced $\pi_1(X^\lambda)$-module structure on $V$. The family~$(\pi_1(X^{\lambda})\longhookrightarrow \pi_1(X))_{\lambda\in \Lambda}$ of injective groupoid maps induces then an isometric isomorphism
\begin{align*}
 H^{\ast}_b(X;V) &\longrightarrow \prod_{\lambda\in\Lambda}^{\|\cdot\|} H^{\ast}_b(X^{\lambda};V^{\lambda})
\intertext{with respect to the product semi-norm, see Definition~\ref{d:prodnorm}. Composing with the isometric isomorphisms given by  Lemma~\ref{l:topconn}, we get an isometric isomorphism}
 H^{\ast}_b(X;V) &\longrightarrow \prod_{\lambda\in\Lambda}^{\|\cdot\|} H^{\ast}_b(X^{\lambda}:\{x(\lambda)\};V_{x(\lambda)})
\end{align*}
for any choice of points $x(\lambda)\in X^\lambda$.
\end{prop}
\begin{proof}
The partition of $X$ into connected components induces a decomposition of~$S_n(\widetilde X)$, $\widetilde X$ and $C_\ast(X)$, which is compatible with the boundary operators and the $\pi_1(X)$-action and preserved under applying $B_{\pi_1(X)}(\;\cdot\;, V)$. The induced semi-norm is exactly the product semi-norm.  
\end{proof}

\begin{prop}\label{p:topcalculatesbc}
Let $X$ be a CW-complex. The following map is an isometric $\R$-cochain isomorphism, natural in $X$
\begin{align*}
R^\ast\colon B(C_\ast^{\sing}(X;\R),\R) &\longrightarrow C_b^{\ast}(X;\R[X])\\
\varphi &\longmapsto \Bigl((\sigma,x)\longmapsto \varphi(p\circ \sigma)\Bigr).
\end{align*} 
\end{prop}
\begin{proof}
 Clearly, for all $n\in\N$ and all $\varphi\in B(C_n^{\sing}(X;\R),\R)$, the map $R^n(\varphi)$ is well-defined, $\pi_1(X)$-invariant and 
\[
 \|R^n(\varphi)\|_{\infty} = \sup_{(\sigma,x)\in Q_n(X)} |\varphi(p\circ \sigma)|= \sup_{\sigma'\in S_n(X)} |\varphi(\sigma')| = \|\varphi\|_{\infty}.
\]
 Thus, $R^n(\varphi)$ is bounded and $R^n$ an isometry. The family $R^\ast$ is also clearly an $\R$-cochain map. An inverse to $R^\ast$ is given by
\begin{align*}
 S^\ast\colon C_b^{\ast}(X; \R[X]) &\longrightarrow B(C_{\ast}^{\sing}(X;\R),\R)\\
\psi &\longmapsto \bigl(\sigma \longmapsto \psi_{\sigma_0}(\widetilde \sigma,\widetilde \sigma_0)\bigr),
\end{align*}
where $\widetilde \sigma$ denotes a lift of $\sigma$. A short calculation shows, that this is indeed a well-defined inverse to $R^\ast$.
\end{proof}

\subsection{The Absolute Mapping Theorem}
 In this section, we will show that there is an isometric isomorphism between the bounded cohomology of a CW-complex and the bounded cohomology of its fundamental groupoid. This will be done by translating Ivanov's proof~\cite{I} of the theorem for groups into the groupoid setting. This is done mainly as a preparation for the proof of the relative mapping theorem. 

Let $X$ be a CW-complex and $V$ a Banach $\pi_1(X)$-module. We will study the Banach $\pi_1(X)$-cochain complex $B(C_{\ast}(X),V)$ together with the canonical~$\pi_1(X)$-augmentation map
\begin{align*}
 \nu \colon V&\longrightarrow B(C_0(X),V)\\
 v &\longmapsto ((\sigma,x)\longmapsto v).
\end{align*}

We will show that $(B(C_{\ast}(X),V),\nu)$ is a strong, relatively injective resolution of~$V$, and then deduce the mapping theorem from the fundamental lemma. In order to prove that this is a strong resolution, we will first discuss how to translate cochain contractions from the group setting into to groupoids.

Recall that for any $\R$-module $V$ and any $x\in \widetilde X$, there is a canonical augmentation map
\begin{align*}
\nu_x\colon  V &\longrightarrow  B(C_{0}^{\sing}(\widetilde X_x;\R),V)\\
v&\longmapsto \bigl(\sigma\longmapsto v\bigr).
\end{align*}
Now, we define pointed equivariant families of cochain contractions in the bounded setting:
\begin{def.} Let $X$ be a CW-complex. Let $(V_x)_{x\in X}$ be a family of Banach $\R$-modules. We call a family of cochain contractions 
 \begin{align*}
  \begin{pmatrix}\bigr(s_x^{\ast} \colon B(C_{\ast}^{\text{sing}}(\widetilde X_x;\R),V_{p(x)})\longrightarrow B(C^{\text{sing}}_{\ast-1}(\widetilde X_x;\R),V_{p(x)})\bigl)_{\ast \in \N_{>0}}\\ s^0_x\colon B(C_{0}^{\sing}(\widetilde X_x;\R),V_{p(x)})\longrightarrow V_{p(x)} \end{pmatrix}_{x\in\widetilde X}
 \end{align*}
of the family of augmented cochain complexes $(B(C_{\ast}^{\text{sing}}(\widetilde X_x;\R),V_{p(x)}),\nu_x)_{x\in \widetilde X}$ \emph{pointed equivariant over $(X,V)$} if 
\begin{enumerate}
\item The family is $\Deck(X)$-equivariant, i.e., for all $\beta\in\Deck(X)$, all $x\in \widetilde X$ and all $\varphi \in B(C_\ast^{\text{sing}}(\widetilde X_x),V_{p(x)})$ 
\[
 s^\ast_{\beta\cdot x}(\beta{\ast}\varphi) = \beta\ast s^\ast_x(\varphi). 
\]
Here, we write $\ast$ to denote the action of $\Deck(X)$ on the cochain complex~$B(C_{\ast}^{\text{sing}}(\widetilde X_x;\R),V_{p(x)})$ given by endowing the module $V_{p(x)}$ with the trivial $\Deck(X)$-action. 
\item They are pointed, i.e., for all $x\in\widetilde X$ and all $\varphi \in B(C_0^{\text{sing}}(\widetilde X_x),V_{p(x)})$
\[
 s_x^{0}(\varphi) =\varphi(x).
\]
\end{enumerate}
\end{def.} 

\begin{rem}\label{r:pointed equivariant}
 Of course, given a family of pointed cochain contractions, one can always find a pointed equivariant family by choosing one contraction for a point in each fibre and defining the other contractions by translation with deck transformations, i.e.: If $(s^\ast_x)_{x\in \widetilde X}$ is a family of pointed cochain contractions, choose a lift $\widetilde {x_0}$ for each $x_0\in X$ and set for each $\beta\in \Deck(X)$ and each~$\varphi\in B(C_\ast^{\text{sing}}(\widetilde X_x),V_{p(x)})$ 
 \[t_{\beta\cdot \widetilde  x_0}^{\ast}(\varphi) = \beta \ast s_{\widetilde x_0}^{\ast}(\beta^{-1}\ast\varphi).\]
Then $(t_x^{\ast})_{x\in \widetilde X}$ is a pointed equivariant family of cochain contractions. 
\end{rem}
\begin{exa}
 Assume $X$ to be aspherical (but not necessarily connected), i.e., assume that the higher homotopy groups of all connected components of~$X$ vanish. Then the space $\widetilde X_x$ is contractible for each $x\in \widetilde X$ and a pointed chain contraction is given by coning with respect to~$x$. Hence by the remark there exists a pointed equivariant family of chain contractions over~$X$.
\end{exa}

\begin{lemma}
 Let $X$ be a CW-complex. For each $x\in \widetilde X$, the map 
\begin{align*}
i_x\colon C_\ast^{\sing}(\widetilde{X}_x) &\longrightarrow C_\ast(X)_{p(x)}\\
 \sigma &\longmapsto (\sigma,x)
\end{align*}
is an isometric chain isomorphism with respect to the restriction of the boundary map to  $C_\ast(X)_{p(x)}$.
\end{lemma}
\begin{proof}
 It is a chain map since the boundary operators act only on the first argument. An inverse is given by the map $(\sigma,y) \longmapsto \beta_y(\sigma)$, where $\beta_y\in\Deck(X_{p(x)})$ is the unique element such that $\beta_y \cdot y = x$. 
\end{proof}

\begin{prop}\label{p:suitabletococontraction}
 Let $X$ be a CW-complex and let $V$ be a Banach $\pi_1(X)$-module. Let $(s_x^\ast)_{x\in\widetilde X}$ be a pointed equivariant family of cochain contractions over~$(X,V)$.  Then 
\begin{align*}
 \fa{n\in \N_{>0}} s^n\colon B(C_n(X),V) &\longrightarrow B(C_{n-1}(X),V)\\
 \varphi &\longmapsto \Bigl( (\sigma,x)\longmapsto s_x^n(\varphi\circ i_x)(\sigma)\Bigr)\\
  s^0\colon B(C_0(X),V) &\longrightarrow V\\
 B(C_0(X),V)_e\ni\varphi &\longmapsto \varphi(\widetilde{e},\widetilde{e})\
\end{align*}
defines a cochain contraction of $(B(C_\ast(X),V),\nu)$. If the $(s^\ast _x)_{x\in\widetilde X}$ are strong, so is $s^\ast$.
\end{prop}
\begin{proof}
For each $\varphi \in B(C_n(X),V)$, the map $s^n(\varphi)$ is well-defined by the equivariance of the family $(s^n_x)_{x\in \widetilde X}$.
The family $(s^n)_{n\in\N}$ is a cochain contraction since $(s_x^{\ast})_{x\in \widetilde X}$ is a family of cochain contractions and $i_x$ is a chain map. In degree 0, we also use that $(s_x^{\ast})_{x\in \widetilde X}$ is pointed: For all $\varphi \in B(C_0(X),V)$  and all $(\sigma,x)\in Q_0(X)$ we have
\begin{align*}
 (s^1\circ\delta^0+ \nu\circ s^0) (\varphi)(\sigma,x)
&= s_x^1( \delta^0 (\varphi\circ i_x))(\sigma) + \varphi(x,x)\\
&= s_x^1(\delta^0 (\varphi\circ i_x))(\sigma) + s^0_x(\varphi\circ i_x)\\
&= (s_x^1\circ \delta^0  + \nu_x\circ s^0_x)(\varphi\circ i_x)(\sigma)\\
&=\varphi(\sigma,x).
\end{align*}
If the $(s^\ast_{x})_{x\in \widetilde X}$ are strong, we note that for all $\varphi \in B(C_n(X),V)$
\begin{align*}
 \|s^n \varphi\|_\infty &= \sup_{(\sigma,x)\in Q_n(X)} |s^n\varphi(\sigma,x) |
= \sup_{(\sigma,x)\in Q_n(X)} |s^n_x(\varphi\circ i_x)(\sigma) |\\
&\leq \sup_{x\in \widetilde X}\|s^n_x(\varphi\circ i_x)\|_{\infty}
\leq  \sup_{x\in \widetilde X}\|(\varphi\circ i_x)\|_{\infty}
\leq \|\varphi\|_\infty. \qedhere
\end{align*}
\end{proof}

\begin{thm}[{\cite{I,Bue11,Loe07}}]\label{t:ivaloeh}
 Let $X$ be a connected CW-complex and $V$ a Banach $\pi_1(X,y)$-module. Then for each $x\in \widetilde X$ there is a strong pointed cochain contraction 
\[
 (s_x^{\ast}\colon B(C^{\sing}_{\ast}(\widetilde X;\R),V')\longrightarrow B(C^{\sing}_{\ast-1}(\widetilde X;\R),V'))_{\ast\in \N}
\]
Here, $V'$ denotes the topological dual of $V$. 
\end{thm}
\begin{proof}[Sketch of proof]
 The main step in Ivanov's proof of the absolute mapping theorem in the group setting is the construction of (partial) strong $\R$-cochain contractions for the cochain complex $B(C_\ast^{\sing}(\widetilde X;\R),\R)$ for $X$ a countable CW-complex~\cite[Theorem 2.4]{I}. This was extended by L\"oh to cochain complexes with twisted coefficients~\cite[Lemma B2]{Loe07} in a dual space and in each case, the cochain contractions can be chosen to be pointed. As noted by B\"uhler, the assumption that $X$ is countable is actually not necessary~\cite{Bue11}. 
\end{proof}

We will use the following simple observation to translate between the fundamental groupoid and the universal cover:
\begin{rem}\label{r:paths}
 Let $x,y\in \widetilde X$ be two points in the same connected component of $\widetilde X$. We write $\gamma_{x,y}$ for the unique element in $\Mor_{\pi_1(X)}(p(x),p(y))$, such that~$\gamma_{x,y}\cdot x = y$. Geometrically, this is given by the projection to $X$ of any path in~$\widetilde X$ from $x$ to $y$. 
By definition, we have for all $\beta\in \Deck(X)$, all~$g\in\pi_1(X)$ with $s(g)= p(y)$ and all $h\in \pi_1(X)$ with $t(h) = p(x)$
\begin{align*}
 \gamma_{\beta\cdot x,\beta\cdot y} &= \gamma_{x,y}\\
g\cdot \gamma_{x,y} &= \gamma_{ x,g\cdot y}\\
\gamma_{x,y}\cdot h &= \gamma_{h^{-1}\cdot x, y}.
\end{align*}

\end{rem}

\begin{prop}\label{p:toprelinj}
Let $X$ be a CW-complex and let $V$ be a Banach $\pi_1(X)$-module. Then for all $n\in \N$, the Banach $\pi_1(X)$-module $B(C_n(X),V)$ is relatively injective.
\end{prop}
\begin{proof}
Let $j\colon U \longrightarrow W$ be a relatively injective map between Banach $\pi_1(X)$-modules. Let $\tau \colon W\longrightarrow U$ be a splitting of $j$ as in Definition~\ref{d:relinj}(i). Let \[\alpha\colon U\longrightarrow B(C_n(X),V)\] be a bounded $\pi_1(X)$-map. We define a family of linear maps by setting for each~$e\in X$
\begin{align*}
 \beta_e\colon W_e&\longrightarrow B(C_n(X)_e,V_e)\\
 w&\longmapsto \bigl ( (\sigma,x) \longmapsto \alpha_e (\gamma_{\sigma_0,x}\cdot \tau_{p(\sigma_0)} (\gamma_{x,\sigma_0}\cdot w))(\sigma,x)\bigr).
\end{align*}
Here $\gamma_{x,\sigma_0}$ denotes the Element in $\pi_1(X)$ that corresponds to a path from~$x$ to $\sigma_0$ in $\widetilde X$. 
\begin{itemize}
\item By Remark~\ref{r:paths}, for each $e\in X$ and each $w\in W_e$, the map $\beta_e(w)$ is well-defined.
\item For each $e\in X$ the map $\beta_e$ is bounded and $\|\beta_e\|_\infty\leq \|\alpha_e\|_\infty$, thus in particular $\|\beta\|_\infty\leq \|\alpha\|_\infty$: We have for each $w\in W_e$ and each $g_0\in\pi_1(X)$ with $t(g_0)=e$
\begin{align*}
 \| \alpha_e(g_0\cdot \tau_{s(g_0)}(g_0^{-1}\cdot w))\|_{\infty}&\leq \|\alpha_e\|_{\infty}\cdot \|g_0\cdot \tau_{s(g_0)}(g_0^{-1}\cdot w)\|_{\infty}\\
&\leq  \|\alpha_e\|_{\infty}\cdot \|\tau_{s(g_0)}\|_\infty \cdot\|(g_0^{-1}\cdot w)\|\\
&\leq  \|\alpha_e\|_{\infty} \cdot\| w\|.
\end{align*} 
\item The map $\beta$ is $\pi_1(X)$-equivariant by Remark~\ref{r:paths}.
\item For all $e\in X$, $w\in W_e$ and $(\sigma,x) \in Q_n(X)_e$ we have
\begin{align*}
 (\beta_e\circ j_e) (w)(\sigma,x) &= \alpha_e(\gamma_{\sigma_0,x} \tau_{p(\sigma_0)}(\gamma_{x,\sigma_0} j_e(w)))(\sigma,x)\\
&= \alpha_e(\gamma_{\sigma_0,x} \tau_{p(\sigma_0)}j_e(\gamma_{x,\sigma_0} (w)))(\sigma,x)\\
&= \alpha_e(w)(\sigma,x).
\end{align*}
Hence $\alpha= \beta\circ j$. \qedhere
\end{itemize}
\end{proof}
\begin{cor}\label{c:strongreltop}
 Let $X$ be a CW-complex and $V$ a Banach $\pi_1(X)$-module. Then $((B(C_{\ast}(X),V')_{\ast\in \N},\nu)$ is a strong, relatively injective~$\pi_1(X)$-resolution of~$V'$.
\end{cor}
\begin{proof}
 By Theorem~\ref{t:ivaloeh}, for each $x\in \widetilde X$, there is a strong, pointed cochain contraction 
\[ (s_x^{\ast}\colon B(C^{\sing}_{\ast}(\widetilde X_x;\R),V'_{p(x)})\longrightarrow B(C^{\sing}_{\ast-1}(\widetilde X_x;\R),V'_{p(x)}))_{\ast\in \N}.
\]
Hence, by Remark~\ref{r:pointed equivariant} and Proposition~\ref{p:suitabletococontraction}, $((B(C_{\ast}(X),V')_{\ast\in \N},\nu)$ is a strong resolution of $V'$. By Proposition~\ref{p:toprelinj}, it is also a relatively injective~$\pi_1(X)$-resolution. 
\end{proof}

\begin{prop}\label{p:mappingthminvers}
 Let $X$ be a CW-complex. There is a canonical, norm non-increasing $\pi_1(X)$-chain map~$\Phi_{\ast}\colon C_\ast(X)\longrightarrow C_\ast(\pi_1(X))$ extending the identity on $\R[X]$, given by
\begin{align*}
 \Phi^X_\ast \colon C_\ast(X) &\longrightarrow C_\ast(\pi_1(X))\\
(\sigma,x)&\longmapsto (\gamma_{\sigma_0,x},\gamma_{\sigma_1,\sigma_0},\dots,\gamma_{\sigma_\ast,\sigma_{\ast-1}}).
\end{align*}
\end{prop}
\begin{proof}
 The map $\Phi^X_{\ast}$ is well-defined and $\pi_1(X)$-equivariant by Remark~\ref{r:paths}. It maps simplices to simplices and is thus norm non-increasing. It is a chain map since for all $(\sigma,x)\in Q_n(X)$ and all $i\in\{0,\dots,n\}$ we have 
\begin{align*}
 \Phi^X_n(\partial_{i,n} \sigma,x) &= (\gamma_{\sigma_0,x},\gamma_{\sigma_1,\sigma_0},\dots,\gamma_{\sigma_i,\sigma_{i-1}}\cdot \gamma_{\sigma_{i+1},\sigma_i},\dots, \gamma_{\sigma_{\ast},\sigma_{\ast-1}})\\
&= \partial_{i,n}\Phi^X_n(\sigma,x).\qedhere
\end{align*}

\end{proof}
\begin{cor}[Absolute Mapping Theorem for Groupoids]\label{c:absmappingthm}
Let $X$ be a CW-complex and let $V$ be a Banach $\pi_1(X)$-module. Then there is a canonical isometric isomorphism of graded semi-normed $\R$-modules
\begin{align*}
 H_b^{\ast}(X;V')\longrightarrow H_b^{\ast}(\pi_1(X);V')
\end{align*}
\end{cor}
\begin{proof}
 By Corollary~\ref{c:strongreltop} the cochain complex $ B(C_{\ast}(X),V')$ is a strong, relatively injective resolution of $V'$. Thus by Proposition~\ref{p:normnonincr}, there exists a norm non-increasing $\pi_1(X)$-cochain map $ B(C_{\ast}(X),V')\longrightarrow B(C_{\ast}(\pi_1(X)),V')$ extending $\id_{V'}$. By Proposition~\ref{p:mappingthminvers}, there exists a norm non-increasing~$\pi_1(X)$-cochain map $B(C_{\ast}(\pi_1(X)),V')\longrightarrow B(C_{\ast}(X),V')$ extending $\id_{V'}$. By the fundamental lemma for groupoids, Proposition~\ref{t:fundamental} these two maps induce canonical, mutually inverse, isometric isomorphisms in bounded cohomology. 
\end{proof}
\section{Relative Bounded Cohomology of Topological spaces}
\sectionmark{Relative Bounded Cohomology of Spaces}
In this section, we prove the relative mapping theorem for certain pairs of CW-complexes, extending the result of Frigerio and Pagliantini~\cite[Proposition 4.4]{FP} to groupoids, and in particular, to the non-connected case.
\subsection{Bounded Cohomology of Pairs of Topological spaces}
In this section, we will define bounded cohomology for $\pi_1$-injective pairs of CW-complexes.
\begin{rem}\label{r:pi1inf}
Let $i\colon A\longrightarrow X$ be a CW-pair, i.e., let $X$ be a CW-complex,~$A$ a CW-subcomplex and $i$ the canonical inclusion.
We call such a map \emph{$\pi_1$-injective} if $\pi_1(i)$ is injective as a groupoid map. Here, this is equivalent to saying that for all $a\in A$, the map $\pi_1(i,a)\colon \pi_1(A,a)\longrightarrow \pi_1(X,a)$ between the fundamental groups at $a$ is injective. 
\end{rem}

Let $i\colon A\longrightarrow X$ be a CW-pair and let $i$ be $\pi_1$-injective. In particular, we can assume $\widetilde A\subset \widetilde X$. Consider the groupoid pair~$\pi_1(i)\colon \pi_1(A)\longhookrightarrow \pi_1(X)$ and let $V$ be a $\pi_1(X)$-module. Then 
\begin{equation*}
  C^\ast(X,A;V) := (B(C_\ast(X),V),B(C_{\ast}(A),\pi_1(i)^\ast V), B(C_\ast(i),V))
\end{equation*}
is a $(\pi_1(X),\pi_1(A))$-cochain complex. We write $K^\ast(X,A;V)$ for the kernel of the map $C_b^{\ast}(i;V)\colon C_b^{\ast}(X;V)\longrightarrow C_b^{\ast}(A;\pi_1(i)^{\ast}V)$. Together with the induced norm, this defines a normed $\R$-cochain complex, see also~Section~\ref{ss:relhomalgpairs}.

We define a category $\ToppBanc$ of CW-pairs $(X,A)$ together with $\pi_1(X)$-modules similarly to $\TopBanc$.  

\begin{rem}[Functoriality]
 Let $i\colon A\longrightarrow X$ and $j\colon B\longrightarrow Y $ be $\pi_1$-injec\-tive CW-pairs, $f\colon (X,A) \longrightarrow (Y,B)$ a continuous map and $\varphi\colon \pi_1(f)^\ast W \longrightarrow V$ a~$\pi_1(X)$-map. Then the right-hand side of the following diagram commutes
\begin{center}
\begin{tikzpicture}
\matrix (m) [scale = 0.5,matrix of math nodes, row sep=2.5em,
column sep=4em, text height=2.5ex, text depth=0.25ex]
{K^\ast(Y,B;W)&C_b^\ast(Y;W)&C_b^\ast(B;\pi_1(j)^\ast W)\\
K^\ast(X,A;V)&C_b^\ast(X;V)&C_b^\ast(A;\pi_1(i)^\ast V)
\\ };
\path[->]
(m-1-1) edge node[auto] {$ $}(m-1-2)
(m-2-1) edge node[auto] {$ $}(m-2-2)
(m-1-2) edge node[auto] {$ C_b^{\ast}(f;\varphi)$}(m-2-2)
(m-1-2) edge node[auto] {$C_b^\ast(j;W)$} (m-1-3)
(m-2-2) edge node[below] {$C_b^\ast(i;V) $} (m-2-3)
(m-1-3) edge node[auto] {$C_b^\ast(f|_{A};\pi_1(i)^\ast\varphi)$} (m-2-3)
;
\path[dashed,->]
(m-1-1) edge node[auto] {$C_b^{\ast}(f;\varphi)|_{K^\ast(Y,B;W)}$}  (m-2-1)
;
\end{tikzpicture}
\end{center}
Hence, the cochain map on the left-hand side is defined. We write 
\[K^\ast(f;\varphi) := C_b^\ast(f;\varphi)|_{K^\ast(Y,B;W)}.\]
This defines a contravariant functor $K^\ast\colon \ToppBanc \longrightarrow \RCh^{\|\cdot\|}$.
\end{rem}

\begin{def.}
 We call 
\[
 H_b^\ast(X,A;V) := H^\ast(K^\ast(X,A;V))
\]
endowed with the induced semi-norm, \emph{the bounded cohomology of $X$ relative to~$A$ with coefficients in $V$}. This defines a functor $\ToppBanc\longrightarrow \Modgrn$. 
\end{def.}
\begin{rem}
 Let $i\colon A\longhookrightarrow X$ be a CW-pair. Gromov~\cite{Gr82} defined bounded cohomology of $X$ relative to $A$ with coefficients in $\R$ as the cohomology of the kernel of the map
\begin{align*}
 B(C^{\sing}_{\ast}(i;\R),\R)\colon B(C^{\sing}_\ast(X;\R),\R) &\longrightarrow B(C_\ast^{\sing}(A;\R),\R),
\end{align*}
endowed with the induced semi-norm. If the pair $(X,A)$ is $\pi_1$-injective, our definition coincides with the one of Gromov by Proposition~\ref{p:topcalculatesbc}. 
\end{rem}
\subsection{The Relative Mapping Theorem}\label{ss:relmap}
\begin{prop}\label{p:relinjtoppair}
 Let $i\colon A\longhookrightarrow X$ be a CW-pair and $V$ a Banach $\pi_1(X)$-module. Then for all $n\in \N$, the Banach $(\pi_1(X),\pi_1(A))$-module $C^n(X,A;V)$ is relatively injective. 
\end{prop}
\begin{proof}
One solves the extension problems as in the proof of Proposition~\ref{p:toprelinj} and sees directly that the constructed maps commute. 
\end{proof}
The following result is due to Frigerio and Pagliantini (for trivial coefficients), extending the construction of Ivanov:
\begin{prop}[\cite{FP}]\label{p:fp} Let $i\colon A\longhookrightarrow X$ be a pair of connected CW-comp\-le\-xes, such that $i$ is $\pi_1$-injective and induces an isomorphism between the higher homotopy groups. Let $V$ be a Banach module. Let $\widetilde i\colon \widetilde A\longrightarrow \widetilde X$ be the inclusion map. Then there exists a  family of norm non-increasing, pointed cochain contractions
 \begin{align*}
  \begin{pmatrix}\bigr( s_x^{\ast} \colon B(C_{\ast}^{\sing}(\widetilde X;\R),V')\longrightarrow B(C^{\sing}_{\ast-1}(\widetilde X;\R),V')\bigl)_{\ast \in \N_{>0}}\\ s^0_x\colon B(C_{0}^{\sing}(\widetilde X;\R),V')\longrightarrow V' \end{pmatrix}_{x\in\widetilde X}
 \end{align*}
and a family of norm non-increasing, pointed cochain contractions
 \begin{align*}
  \begin{pmatrix}\bigr(\hat s_a^{\ast} \colon B(C_{\ast}^{\sing}(\widetilde A;\R),V')\longrightarrow B(C^{\sing}_{\ast-1}(\widetilde A;\R),V')\bigl)_{\ast \in \N_{>0}}\\ \hat s^0_a\colon B(C_{0}^{\sing}(\widetilde A;\R),V')\longrightarrow V' \end{pmatrix}_{a\in\widetilde A}
 \end{align*}
that is compatible with the restriction to $\widetilde A$, i.e., the following diagram commutes for all $a\in\widetilde A$:
\begin{center}
\begin{tikzpicture}
\matrix (m) [scale = 0.5,matrix of math nodes, row sep=3.5em,
column sep=3em, text height=1.5ex, text depth=0.25ex]
{B(C_{\ast}^{\sing}(\widetilde X),V')& B(C^{\sing}_{\ast-1}(\widetilde X),V')\\
B(C_{\ast}^{\sing}(\widetilde A),V')& B(C^{\sing}_{\ast-1}(\widetilde A),V')
\\ };
\path[->]
(m-1-1) edge node[auto] {$s_{\widetilde i(a)}^{\ast}$}(m-1-2)
(m-1-1) edge node[left] {$B(C_{\ast}^{\sing}(\widetilde i),V')$}(m-2-1)
(m-2-1) edge node[auto] {$\hat s_a^{\ast}$}(m-2-2)
(m-1-2) edge node[auto] {$B(C_{\ast-1}^{\sing}(\widetilde i),V')$}(m-2-2)
;
\end{tikzpicture}
\end{center}
\end{prop}

\begin{rem}
  The proof of Proposition~\ref{p:fp} is a generalisation of Ivanov's proof of the fact that $B(C_{\ast}^{\sing}(\widetilde X;\R),\R)$ is a strong resolution of $\R$. Frigerio and Pagliantini state Proposition~\ref{p:fp} only for trivial coefficients. The extension to dual space coefficients and pointed cochain contractions is straightforward~\cite[Proposition A.8]{Bl14}. Park~\cite[Lemma 4.2]{Pa03} stated without proof that, based on Ivanov's work,  Proposition~\ref{p:fp} also holds without the assumption on the higher homotopy groups. Pagliantini~\cite[Remark 2.29]{Pa13} demonstrates however, that Ivanov's proof cannot be generalised directly to the relative setting.
\end{rem}

\begin{lemma}\label{l:suitablepair}
 We can assume the family of cochain contractions $(s^\ast_x)_{x\in\widetilde X}$ and~$(\hat s^\ast_a)_{a\in\widetilde A}$ in Proposition~\ref{p:fp} to be pointed equivariant. 
\end{lemma}
\begin{proof}
Since $i$ is $\pi_1$-injective, we can identify $\Deck(A)$ with the subgroup of deck transformations in $\Deck(X)$ mapping $\widetilde A$ to itself. We proceed as in Remark~\ref{r:pointed equivariant}: For each $x_0\in X$, choose a lift $\widetilde {x_0}\in \widetilde X$, such that $\widetilde x_0\in \widetilde A$ if~$x_0\in A$. Set for each~$x_0\in X$, $\beta\in \Deck(X)$ and each~$\varphi\in B(C_\ast^{\text{sing}}(\widetilde X),V')$ 
 \begin{align*}
  t_{\beta\cdot \widetilde  x_0}^{\ast}(\varphi) &= \beta \ast s_{\widetilde x_0}^{\ast}(\beta^{-1}\ast\varphi)
\intertext{and similarly for each  $a_0\in A$, $\beta\in \Deck(A)$ and each~$\varphi\in B(C_\ast^{\text{sing}}(\widetilde A),V')$}
 \hat t_{\beta\cdot \widetilde  a_0}^{\ast}(\varphi) &= \beta \ast \hat s_{\widetilde a_0}^{\ast}(\beta^{-1}\ast\varphi)
 \end{align*}
Then $(t_x^{\ast})_{x\in \widetilde X}$ and $(\hat t_a^\ast)_{a\in\widetilde A}$ are  pointed equivariant families of cochain contractions. It is easy to see that  $(t_x^{\ast})_{x\in \widetilde X}$ and $(\hat t_a^\ast)_{a\in\widetilde A}$ are still compatible with the restriction to~ $\widetilde A$.
\end{proof}

\begin{cor}\label{c:strongtoppair}
  Let $i\colon A\longhookrightarrow X$ be a CW-pair, such that $i$ is $\pi_1$-injective and induces an isomorphism between the higher homotopy groups on each connected component of $A$. Furthermore, let $V$ be a Banach $\pi_1(X)$-module. Then~$C^\ast(X, A;V')$ is a strong $(\pi_1(X),\pi_1(A))$-resolution of $V'$
\end{cor}
\begin{proof}

 By Lemma~\ref{l:suitablepair}, there are pointed equivariant families of norm non-increasing cochain contractions for $X$ and $A$. By a short calculation, we see that the corresponding norm non-increasing cochain contractions for $X$ and $A$ constructed in Proposition~\ref{p:suitabletococontraction} commute~\cite[Corollary 5.3.9]{Bl14}.
\end{proof}

\begin{lemma}\label{l:mappinginverspairs}
Let $i\colon A\longrightarrow X$ be a $\pi_1$-injective CW-pair, and let $V$ be a Banach $\pi_1(X)$-module. Then there exists a canonical, norm non-increa\-sing~$(\pi_1(X),\pi_1(A))$-cochain map
\[
 C^\ast(\pi_1(X),\pi_1(A);V) \longrightarrow C^\ast(X,A;V)
\]
extending $(\id_V,\id_{\pi_1(i)^\ast V})$. 
\end{lemma}
\begin{proof}
 The map is given given by $(B(\Phi^X_\ast,V),B(\Phi^A_\ast,\pi_1(i)^\ast V))$, where the morphisms~$\Phi_{\ast}^X$ and~$\Phi_{\ast}^A$ are as in Proposition~\ref{p:mappingthminvers}. Its easy to see that this is a~$(\pi_1(X),\pi_1(A))$-map. 
\end{proof}

\begin{thm}[Relative Mapping Theorem]\label{t:relmap} Let $i\colon A\longhookrightarrow X$ be a CW-pair, such that $i$ is $\pi_1$-injective and induces isomorphisms between the higher homotopy groups on each connected component of $A$. Let $V$ be a Banach $\pi_1(X)$-module. Then there is a canonical isometric isomorphism 
\begin{align*}
 H_b^\ast(X,A;V')&\longrightarrow H_b^\ast(\pi_1(X),\pi_1(A);V').
\end{align*}
\end{thm}
\begin{proof}
 By Proposition~\ref{p:relinjtoppair} and Corollary~\ref{c:strongtoppair}, the $(\pi_1(X),\pi_1(A))$-cochain complex $C^\ast(X,A;V')$ is a strong, relatively injective $(\pi_1(X),\pi_1(A))$-reso\-lution of $V'$. Therefore, by Proposition~\ref{p:fundamentalpairs}, there exists a norm non-increasing~$(\pi_1(X),\pi_1(A))$-cochain map 
\[
 \alpha^\ast\colon C^\ast(X,A;V')\longrightarrow C^\ast(\pi_1(X),\pi_1(A);V')
\]
extending $(\id_{V'},\id_{\pi_1(i)^\ast V'})$. By Lemma~\ref{l:mappinginverspairs}, there is a norm non-increasing $(\pi_1(X),\pi_1(A))$-cochain map $ C^\ast(\pi_1(X),\pi_1(A);V') \longrightarrow C^\ast(X,A;V')$ extending $(\id_{V'},\id_{\pi_1(i)^\ast V'})$. By the fundamental lemma for pairs, Proposition~\ref{p:fundamentalpairs}, these maps induce mutually inverse, norm non-increasing isomorphisms in bounded cohomology. So the map $H_b^\ast(X,A;V')\longrightarrow H_b^\ast(\pi_1(X),\pi_1(A);V')$ induced by $\alpha^\ast$ is an isometric isomorphism. Also by the fundamental lemma, this isomorphism doesn't depend on the extension of $(\id_{V'},\id_{\pi(i)^\ast V'})$. 
\end{proof}

\begin{cor}\label{c:relmappingthm}
 Let $i\colon A\longhookrightarrow X$ be a CW-pair, such that $i$ is $\pi_1$-injective and induces isomorphisms between the higher homotopy groups on each connected component of $A$. Let $V$ be a Banach $\pi_1(X)$-module. Let  $\pi_1(A)$ be amenable. Then there is a canonical isometric isomorphism 
\begin{align*}
 H_b^\ast(X,A;V')&\longrightarrow H_b^\ast(X;V').
\end{align*}
\end{cor}
Corollary~\ref{c:relmappingthm} was stated by Gromov~\cite{Gr82} without the assumptions on~$i$ to be~$\pi_1$-injective and to induce isomorphisms between the higher homotopy groups, but without a proof for the map to be isometric. This has been one motivation for us to study relative bounded cohomology in the groupoid setting in the first place. There is now, however, a short and beautiful proof of this stronger result by Bucher, Burger, Frigerio, Iozzi, Pagliantini and Pozzetti~\cite[Theorem 2]{Petal13}. This has also been shown independently by Kim and Kuessner~\cite[Theorem 1.2]{KK12} via multicomplexes. 
\begin{proof}The following diagram commutes:
 \begin{center}
\begin{tikzpicture}
\matrix (m) [scale = 0.5,matrix of math nodes, row sep=3.5em,
column sep=3em, text height=1.5ex, text depth=0.25ex]
{H_b^{n}(X,A;V')& H_b^{n}(X;V')\\
H_b^{n}(\pi_1(X),\pi_1(A);V')& H_b^n(\pi_1(X);V')
\\ };
\path[->]
(m-1-1) edge node[auto] {$ $}(m-1-2)
(m-1-1) edge node[left] {$\cong $}(m-2-1)
(m-2-1) edge node[below] {$\cong $}(m-2-2)
(m-1-2) edge node[auto] {$\cong $}(m-2-2)
;
\end{tikzpicture}
\end{center}
Here, the column maps are the isometric isomorphisms induced by the (topological) mapping theorem, Theorem~\ref{t:relmap}, and the row maps are induced by the canonical inclusions. The lower row map is an isometric isomorphisms by the algebraic mapping theorem, Corollary~\ref{c:algebraicmappingthm}.
\end{proof}
\begin{rem}
One important reason to consider relative bounded cohomology  is to study manifolds with boundary relative to the boundary, due to the relation between bounded cohomology and simplicial volume. Let~$(M,\partial M)$ be an aspherical manifold with $\pi_1$-injective aspherical boundary and let $N$ be a manifold without boundary. Then any $N$-bundle over~$M$ satisfies the condition of the relative mapping theorem, Corollary~\ref{c:relmappingthm}.  This gives many new examples for which the relative mapping theorem holds, since there are important examples of such pairs $(M,\partial M)$. We mention just some that are interesting with respect to the simplicial volume:
\begin{enumerate}
\item Compact aspherical 3-manifolds relative to a union of incompressible boundary components~\cite{AFW12}. 
\item The relative hyperbolisation construction of Davis, Januszkiewicz and Weinberger~\cite{DJW01} gives rise to many exotic examples. Let $X$ be a manifold with boundary $Y$ and assume that each connected component of $Y$ is aspherical. Then the relative hyperbolisation $J(X,Y)$ relative to $Y$ satisfies the assumption of Corollary~\ref{c:relmappingthm}.
\item Compact hyperbolic manifolds with totally geodesic boundary relative to the boundary~\cite[Proposition 13.1]{Bele07}.
\end{enumerate}

\end{rem}

\bibliographystyle{plain}
\bibliography{/home/matthias/Documents/Bib/refs}

\end{document}